\documentclass[twoside]{article}

\usepackage{amsmath}    
\usepackage{amssymb}    
\usepackage{amsthm}     
\usepackage{bbm}
\usepackage{mathtools}  
\usepackage{hyperref}
\usepackage{cleveref}
\usepackage{subcaption}
\usepackage{booktabs}

\usepackage{bm}         
\usepackage{mathrsfs}   
\usepackage{enumitem}   
\usepackage{amsthm}
\usepackage{xcolor}
\usepackage{enumitem}
\usepackage{natbib}

\usepackage{algorithm, algorithmic}          
\usepackage{float}              

\renewcommand{\d}{\mathrm{d}}

\newtheorem{theorem}{Theorem}
\newtheorem{assumption}{Assumption}

\newtheorem{proposition}{Proposition}

\newtheorem{lemma}{Lemma}
\newtheorem{corollary}{Corollary}

\DeclareMathOperator*{\argmin}{arg~min}
\DeclareMathOperator*{\argmax}{arg~max}

\usepackage[preprint]{aistats2026}
\begin{document}

%

%

\twocolumn[

\aistatstitle{\texttt{Conformal-DRO}: Distributionally Robust Optimization with Conformalized Ambiguity Set}

\aistatsauthor{Luhao Zhang \And Shixiang Zhu}

\aistatsaddress{Johns Hopkins University \And Carnegie Mellon University} ]

\begin{abstract}
  Data-driven distributionally robust optimization (DRO) typically treats the conditional outcome law as fixed and uses ambiguity sets to capture estimation error. This paper studies latent distributional heterogeneity, where each instance has an unobserved law but contributes only one observation, so uncertainty persists even if the mixture law is known. We propose \texttt{Conformal-DRO}, which uses nested conformal regions to construct an ambiguity set for the future latent law. Under exchangeability, the set covers this law with probability at least $1-\alpha$ in finite samples, without estimating underlying latent laws or their mixing mechanism. The conformal path induces a data-driven transport geometry, while $\alpha$ determines the radius. The worst-case problem reduces to a finite linear program over conformal shells and admits sparse adversarial solutions. The resulting robust value provides a finite-sample certificate for the selected decision's expected cost.
\end{abstract}

\section{Introduction}
\label{sec:introduction}
\vspace{-.06in}

Many decisions must be made after observing contextual information but before the uncertain outcome is realized. 
A physician chooses a treatment after reviewing a patient's records \citep{mo2021learning}; a system operator schedules resources after observing current operating conditions \citep{shi2022day}. 
Data-driven distributionally robust optimization (DRO) protects such decisions by optimizing against a collection of plausible outcome distributions rather than committing to a single estimated model 
\citep{delage2010distributionally,wiesemann2014distributionally,rahimian2022frameworks}. 

The effectiveness of DRO depends critically on how this ambiguity set is constructed. 
In transport-based DRO, two choices are especially consequential: the \emph{transport geometry}, which determines the relative cost of redistributing probability mass across outcomes, and the \emph{radius}, which controls the extent of the permitted redistribution. 
A poorly chosen geometry may hedge against irrelevant perturbations, whereas an improperly calibrated radius can produce either fragile or overly conservative decisions. 
Canonical transport-based constructions typically specify a ground cost, often induced by a norm, and calibrate the radius using concentration bounds, asymptotic approximations, or validation
\citep{mohajerin2018data,blanchet2019quantifying}.

These constructions generally adopt a \emph{fixed-law} interpretation: the relevant outcome distribution is unknown but nonrandom, and the ambiguity set primarily captures the sampling error involved in estimating it. 
In contextual extensions, this amounts to associating each context $x$ with a single conditional law. 
As that law is estimated more accurately, the corresponding ambiguity set is expected to contract.

In many contextual problems, however, the distribution relevant to the next decision is itself uncertain \citep{efron2024empirical}. 
Even patients with the same recorded characteristics may have different treatment-response distributions because relevant biological factors, such as disease subtypes or drug sensitivities, remain unobserved \citep{dahabreh2016using}. 
We refer to the instance-specific outcome distribution as its \emph{latent conditional law}. 
Because only one outcome is observed from each such law, pooling observations across instances targets the \emph{mixture conditional law}, the average of their latent laws, rather than the law governing the next instance. 
Thus, even if the mixture law were known exactly, uncertainty about the next latent law would remain. 
An ambiguity set shrinking around the mixture law may therefore become increasingly precise about population-average behavior while failing to capture this instance-level heterogeneity.

This distinction changes the ambiguity-set design problem. 
Rather than accounting only for estimation error around a fixed conditional law, the ambiguity set should contain the latent law governing the next instance with a prescribed probability. 
Its geometry should encode which reallocations of probability mass are statistically compatible with the observed data, while its radius should reflect the desired confidence in covering the future law. 
Constructing such a set is challenging because the underlying latent laws are never observed, only one outcome is available from each of them, and their across-instance distribution---the \emph{mixing law}---is unknown. 
Directly estimating or comparing the latent laws is therefore impossible without substantial structural assumptions.

Our key observation is that \emph{validity guarantees in the observable outcome space provide indirect information about these unobserved laws} \citep{zheng2026beyond}. 
Conformal prediction, for example, constructs a set of plausible outcomes that contains a future outcome with a prescribed finite-sample miscoverage rate $\gamma$ under exchangeability
\citep{vovk2005algorithmic,lei2018distribution,romano2019conformalized,angelopoulos2023conformal}. 
For any candidate latent law, the probability mass it assigns outside the conformal set is precisely its conditional probability of generating an uncovered outcome. 
Averaging this ``leakage'' over future instances yields the observable conformal miscoverage probability. 
This conformal measure inversion therefore lifts outcome-level validity to a moment bound on a functional of the future latent law, without estimating either the historical latent laws or their mixing mechanism. 
The resulting normalized leakage, which we call \emph{conformal incompatibility}, quantifies how strongly a candidate law conflicts with the conformal evidence; Figure~\ref{fig:overview} illustrates this construction.

\begin{figure}[!t]
    \centering
    \includegraphics[width=\linewidth]{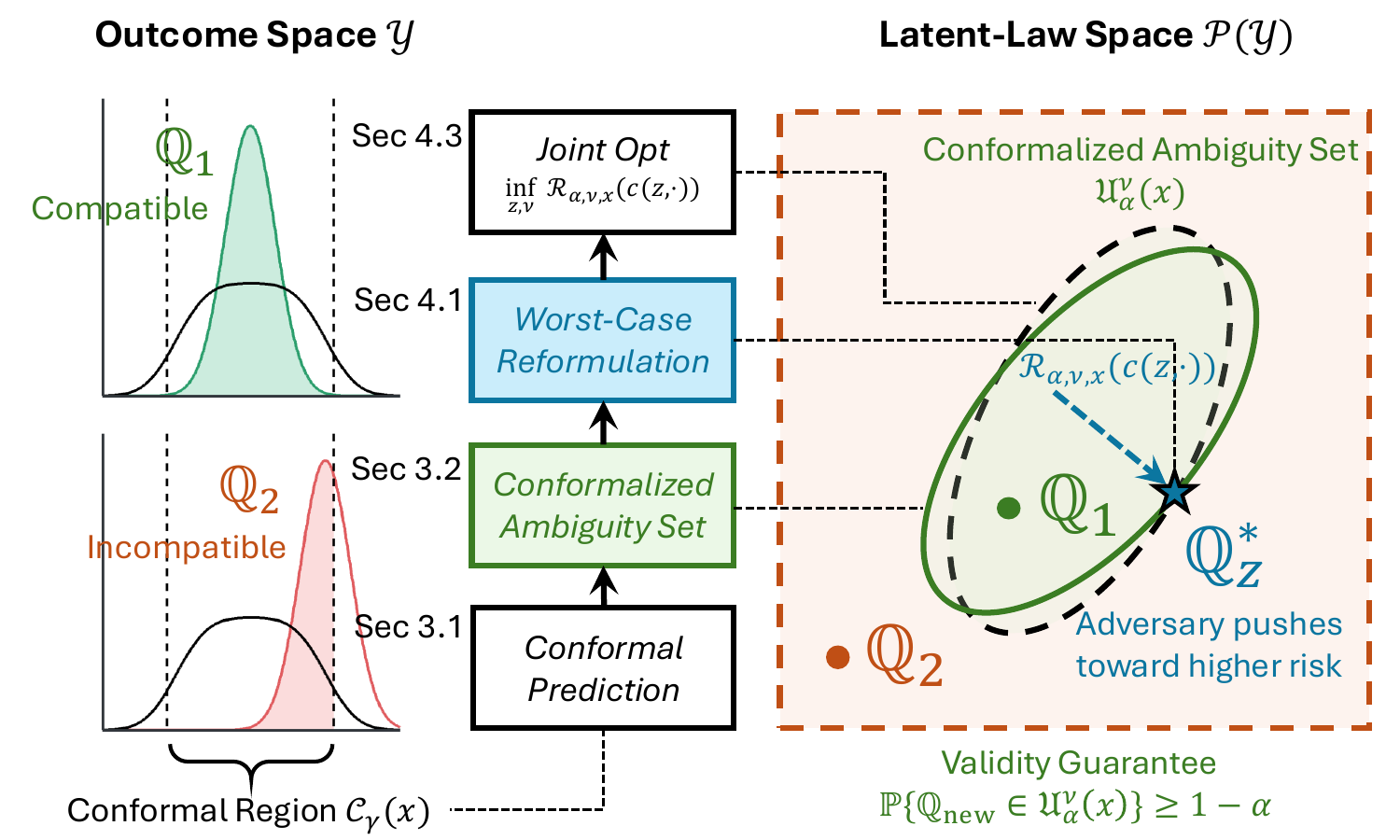}
    \caption{Overview of \texttt{Conformal-DRO}. In the outcome space, the conformal region $\mathcal C_\gamma(x)$ partitions outcomes into two subregions. Candidate latent laws are assessed by how they allocate probability mass across this partition: $\mathbb Q_1$ is retained, whereas $\mathbb Q_2$, which assigns excessive mass outside the conformal region, is excluded. This outcome-level information is lifted to the conformalized ambiguity set $\mathfrak U_\alpha^\nu(x)$ in the latent-law space. For a fixed decision $z$, the adversary selects a compatible law $\mathbb Q_z^\star$ that maximizes risk, while joint optimization chooses the decision and ambiguity-set geometry $\nu$ to minimize the resulting worst-case risk.}
    \label{fig:overview}
\end{figure}

We combine these constraints across a nested conformal path to construct a conformalized ambiguity set. 
The nested regions partition the outcome space into shells of increasing incompatibility. 
Weights assigned to different conformal levels determine the incompatibility penalty associated with each shell and can be tuned to emphasize regions that are consequential for the downstream decision. 
The conformal path and its weights determine the transport geometry, while the desired latent-law miscoverage level $\alpha$ determines the radius $1/\alpha$.
The framework therefore gives both central design choices in transport-based DRO direct statistical interpretations.

Because the incompatibility penalty is constant within each conformal shell, the adversary only needs to allocate probability mass across shells and concentrate that mass near an outcome attaining the largest loss. 
The worst-case expectation therefore reduces to a finite linear program over shell probabilities, and an optimal adversarial law can be chosen with positive mass in at most two shells. 
Thus, conformal inference supplies statistically valid law-level constraints, while the DRO objective converts them into a finite and interpretable probability-allocation problem. 

Our contributions are threefold:
($i$) We introduce a principled ambiguity-set construction for contextual DRO under latent distributional heterogeneity, in which the conformal path determines the transport geometry and the desired coverage level determines the radius.
($ii$) We adapt conformal measure inversion to distribution-valued targets, translating observable outcome-level validity into a distribution-free ambiguity set with finite-sample marginal coverage of the future latent conditional law and a corresponding expected-cost certificate.
($iii$) We characterize the computational structure of the resulting DRO problem, including its reduction to probability allocation across conformal shells, the sparsity of worst-case laws, and the joint design of decisions and ambiguity-set geometry under a resolution budget.

Our numerical experiments demonstrate that the proposed method maintains finite-sample validity while achieving competitive regret. They also show that the learned geometry adapts to different distributional shapes and that increasing its resolution yields diminishing improvements in the robust value.

\paragraph{Related Work}

Contextual stochastic optimization uses covariates to estimate conditional distributions or learn decisions directly, with distributionally robust extensions constructing context-dependent, residual-based, or transport-based ambiguity sets \citep{ban2019big,bertsimas2020predictive,elmachtoub2022smart,kannan2024residuals,esteban2022distributionally,van2022robust,yang2026decision}. Like conventional data-driven DRO \citep{delage2010distributionally,wiesemann2014distributionally,mohajerin2018data,blanchet2019quantifying}, these methods generally treat the relevant conditional law as fixed and use the ambiguity set to quantify estimation error, so the set is expected to contract with more data. We instead allow each instance to have its own latent conditional law: the observable conditional distribution averages over these laws, while the future target is a new random law drawn from their population. This complements \citet{yang2026decision}, who preserve conditional structure while perturbing a nominal joint law, whereas we model heterogeneity remaining after conditioning on the observed context. Although this hierarchical formulation resembles empirical Bayes \citep{efron2024empirical}, we estimate neither the mixing law nor a posterior; instead, we construct a conformal ambiguity set for the future random law, with a data-induced transport geometry and a radius determined by the desired coverage level.

Conformal prediction provides distribution-free, finite-sample marginal coverage under exchangeability \citep{vovk2005algorithmic,lei2018distribution,romano2019conformalized,angelopoulos2023conformal}, while optimization-oriented approaches typically embed outcome-space conformal regions into robust or constrained decision problems \citep{patel2024conformal,zheng2024generative,zhang2026topology,chen2026learning,zhu2022distributionally}. Among the closest works, \citet{chen2026learning} learns decision-tailored conformal uncertainty sets over future outcomes, whereas our uncertainty object is a latent probability law. 
\citet{zheng2026beyond} introduces a related measure-inversion principle for parametric forward models, but focuses on explicitly recovering a confidence set in parameter or distribution space, which can require searching over the full candidate space. Our approach instead uses measure inversion to induce an ambiguity set whose downstream DRO problem admits a tractable reformulation, without explicitly enumerating or constructing the set. Separately, \citet{patel2023variational} constructs conformal regions for latent parameters, while \citet{patel2024non} calibrates Wasserstein balls around amortized posterior approximations using reference posterior samples. In contrast, we observe only one outcome from each latent law, require neither a known simulator nor reference posterior samples, and estimate neither the mixing distribution nor individual latent laws. Our goal is instead a finite-sample-valid ambiguity set over probability laws tailored to downstream optimization.


\vspace{-.06in}
\section{Problem Setup}
\vspace{-.06in}

Following the standard contextual stochastic optimization framework \citep{qi2022integrating}, a decision maker observes a context $X=x$ and chooses a decision $z$ before the uncertain outcome $Y$ is realized. 
Let $\mathcal X$ and $\mathcal Y$ be standard Borel spaces, let $\mathcal P(\mathcal Y)$ denote the set of probability measures on $\mathcal Y$, and let $\mathcal Z(x)$ be the feasible decision set. Given a cost function $c(z,y)$, knowledge of the conditional outcome distribution would yield the classical contextual stochastic optimization problem
\begin{equation*}
z^\star(x)\in\operatorname*{arg\,min}_{z\in\mathcal Z(x)}\mathbb E\!\left[c(z,Y)\mid X=x\right].
\vspace{-0.15in}
\end{equation*}

We instead consider an empirical Bayes setting in which the conditional law varies across instances because of heterogeneity that remains unobserved after conditioning on $X$ \citep{efron2024empirical}. 
For each instance $i$, let $X_i \in \mathcal X$ be an observed context, $Y_i \in \mathcal Y$ an observed outcome, and $\mathbb Q_i(\cdot \mid X_i)\in\mathfrak \mathcal{P}(\mathcal{Y})$ an unobserved conditional law satisfying,
\begin{equation}
\label{eq:hierarchical-model}
    \mathbb Q_i\mid X_i\sim\pi(\cdot\mid X_i),
    \quad
    Y_i\mid(X_i,\mathbb Q_i)\sim\mathbb Q_i(\cdot\mid X_i),
\end{equation}
where the mixing kernel $\pi$ is unknown and is not estimated. 
We suppress the dependence on $X_i$ and write $\mathbb Q_i$ throughout for notational simplicity.

Our goal is to prescribe a robust decision for a future context $X_{n+1}$ using historical data $\mathcal D_n \coloneqq (X_i, Y_i)_{i=1}^n$. Because the corresponding latent law $\mathbb Q_{n+1}$ remains uncertain even after $X_{n+1}$ is observed, we address this uncertainty through data-driven DRO. Given an ambiguity set $\mathfrak U_\alpha(x;\mathcal D_n)\subseteq\mathcal P(\mathcal Y)$, we consider
\begin{equation}
\label{eq:original-dro}
\widehat z_n(x)\in\operatorname*{arg\,min}_{z\in\mathcal Z(x)}\sup_{\mathbb Q\in\mathfrak U_\alpha(x;\mathcal D_n)}\int_{\mathcal Y}c(z,y)\,\mathbb Q(\mathrm dy),
\end{equation}
where $\alpha\in(0,1)$ denotes the desired latent-law miscoverage level. We seek to construct $\mathfrak U_\alpha$ such that
$$
\mathbb P\!\left\{\mathbb Q_{n+1}\in\mathfrak U_\alpha(X_{n+1};\mathcal D_n)\right\}\geq 1-\alpha.
$$
The key challenge is that each latent law $\mathbb Q_i$ generates only one observation $Y_i$, precluding direct estimation of the individual laws or calibration of an ambiguity set in $\mathcal P(\mathcal Y)$ without additional structural assumptions.

To address this, we develop a distribution-free approach that imposes no parametric or structural assumptions on the family of latent laws or their mixing mechanism and does not attempt to estimate these distributions directly. Instead, we first use observable conformal miscoverage guarantees to construct a finite-sample-valid ambiguity set for the future latent law. Then we exploit the nested conformal path as a data-driven geometry, yielding a finite-dimensional shell reformulation that jointly optimizes the ambiguity geometry and the downstream decision.

\vspace{-.06in}
\section{Valid Ambiguity Set Construction}
\label{sec:valid-ambiguity-set}
\vspace{-.06in}

\subsection{Preliminaries: Conformal Prediction}
\vspace{-.06in}

We first establish the classical outcome-level conformal guarantee underlying our construction. Split conformal prediction begins with a fitted predictor $\widehat f$, which may be any black-box predictive model. The predictor is trained on data independent of the calibration data $\mathcal D_n$ and the next instance $(X_{n+1}, Y_{n+1})$.

\begin{assumption}[Exchangeability]
\label{ass:exchangeability}
The latent pairs 
\[
(X_1,\mathbb Q_1),\ldots,(X_{n+1},\mathbb Q_{n+1})
\]
are exchangeable. Conditional on these pairs, the outcomes are independent, with
$
Y_i \sim \mathbb{Q}_i.
$
\end{assumption}
Assumption~\ref{ass:exchangeability} implies that the pairs $(X_i,Y_i)_{i=1}^{n+1}$ are exchangeable. 
Let $s_{\widehat{f}}:\mathcal X\times\mathcal Y\to\mathbb R$ be a measurable nonconformity score and define $S_i=s_{\widehat{f}}(X_i,Y_i)$ for all $i$.
For $\gamma\in(0,1)$, define
$
    k_\gamma=\left\lceil(n+1)(1-\gamma)\right\rceil
$
and let $\widehat q_{1-\gamma}$ be the $k_\gamma$-th order statistic of $S_1,\ldots,S_n$, with $\widehat q_{1-\gamma}=+\infty$ when $k_\gamma=n+1$. The split-conformal prediction set is
\begin{equation}
\label{eq:conformal-set}
    \mathcal{C}_\gamma(x;\mathcal D_n)
    =
    \{y\in\mathcal Y:s_{\widehat{f}}(x,y)\leq\widehat q_{1-\gamma}\}.
\end{equation}

\begin{lemma}[Outcome-level validity \citep{vovk2005algorithmic}]
\label{lem:outcome-validity}
Under Assumption~\ref{ass:exchangeability}, for every fixed $\gamma\in(0,1)$,
\[
    \mathbb P\{Y_{n+1}\in\mathcal{C}_\gamma(X_{n+1})\}
    \geq 1-\gamma.
\]
\end{lemma}
Whenever no ambiguity arises, we suppress the dependence on the calibration set $\mathcal D_n$ and write, for example, $\mathcal C_\gamma(X_{n+1})$ in place of $\mathcal C_\gamma(X_{n+1};\mathcal D_n)$.

\vspace{-.06in}
\subsection{Conformalized Ambiguity Set}
\vspace{-.06in}

We next lift the outcome-level guarantee in Lemma~\ref{lem:outcome-validity} to the space of probability measures. For a candidate law $\mathbb Q\in\mathcal P(\mathcal Y)$, the quantity $1-\mathbb Q(\mathcal C_\gamma(x))$ is the probability mass that $\mathbb Q$ assigns outside the level-$\gamma$ conformal region. Under the hierarchical model in \eqref{eq:hierarchical-model}, the tower property connects this distribution-level ``leakage'' to the observable conformal miscoverage event:
\[
    \mathbb E\!\left[
        1-\mathbb Q_{n+1}\!\left(
            \mathcal C_\gamma(X_{n+1})
        \right)
    \right]
    =
    \mathbb P\!\left\{
        Y_{n+1}\notin
        \mathcal C_\gamma(X_{n+1})
    \right\}
    \leq\gamma.
\]
This identity is the key \emph{measure-inversion} step: conformal prediction controls an observable outcome-level event, $Y_{n+1}\notin \mathcal C_\gamma(X_{n+1})$, while the hierarchical model in \eqref{eq:hierarchical-model} converts that guarantee into a moment bound on an unobserved probability measure $\mathbb{Q}_{n+1}$. It therefore provides finite-sample information about the future latent law without requiring the historical latent laws or their mixing kernel to be observed or estimated.

Formally, define the level-$\gamma$ \emph{conformal incompatibility}:
\[
    q_\gamma(\mathbb Q;x)
    \coloneqq
    \frac{
        1-\mathbb Q\!\left(\mathcal C_\gamma(x)\right)
    }{\gamma}.
\]
The score quantifies the degree to which a candidate law $\mathbb Q$ is incompatible with the conformal region $\mathcal C_\gamma(x)$: larger values correspond to less probability mass assigned to the region. Here, $q_\gamma$ is an $e$-statistic \citep{vovk2021values}, satisfying
\begin{equation}
\label{eq:e-moment}
\mathbb E\!\left[
q_\gamma\!\left(
\mathbb Q_{n+1};X_{n+1}
\right)
\right]
\leq 1.
\end{equation}

A single conformal level distinguishes only between outcomes inside and outside one region. To obtain a more refined description of how a candidate law distributes mass along the multiple conformal levels, let $\Gamma\subset(0,\alpha)$ be a collection of outcome-level miscoverage values, where $\alpha\in(0,1)$ is the desired latent-law miscoverage level. For a probability measure $\nu$ supported on $\Gamma$, define the \emph{aggregated conformal incompatibility}
\begin{equation}
\label{eq:aggregate-score}
    e_\nu(\mathbb Q;x)
    \coloneqq
    \int_\Gamma
        q_\gamma(\mathbb Q;x)
        \,\nu(\d\gamma).
\end{equation}
We call $\nu$ the \emph{conformal-level aggregation rule}; it determines which portions of the conformal path receive greater emphasis. For any fixed $\nu$, Tonelli's theorem and \eqref{eq:e-moment} give
\begin{equation}
\label{eq:aggregate-moment}
    \mathbb E\!\left[
        e_\nu\!\left(
            \mathbb Q_{n+1};X_{n+1}
        \right)
    \right]
    \leq 1.
\end{equation}

Motivated by this moment bound, we define the \emph{conformalized ambiguity set}
\[
    \mathfrak U_\alpha^\nu(x)
    \coloneqq
    \left\{
        \mathbb Q\in\mathcal P(\mathcal Y):
        e_\nu(\mathbb Q;x)
        \leq\frac{1}{\alpha}
    \right\}.
\]
The threshold $1/\alpha$ converts the moment bound in~\eqref{eq:aggregate-moment} into a containment guarantee through Markov's inequality. The restriction $\Gamma\subset(0,\alpha)$ is not required for validity, but it ensures that a single-level rule $\nu=\delta_\gamma$ imposes the nonvacuous leakage restriction $1-\mathbb Q(\mathcal C_\gamma(x))\leq\gamma/\alpha<1$.

Because $e_\nu(\mathbb Q;x)$ is affine in $\mathbb Q$, the ambiguity set $\mathfrak U_\alpha^\nu(x)$ is a convex generalized-moment set. The construction does not identify the full latent law; instead, it constrains how that law allocates probability mass across regions of increasing atypicality. Incorporating additional conformal levels refines this resolution, while $\nu$ determines which distinctions along the conformal path are most consequential.

\begin{theorem}[Finite-sample latent-law coverage]
\label{thm:multilevel-validity}
Suppose Assumption~\ref{ass:exchangeability} holds and the conformal-level aggregation rule $\nu$ is fixed independently of the final calibration sample $\mathcal D_n$ and the future observation. Then
\begin{equation}
\label{eq:conditional-law-coverage}
    \mathbb P\!\left\{
        \mathbb Q_{n+1}
        \in
        \mathfrak U_\alpha^\nu(X_{n+1})
    \right\}
    \geq 1-\alpha.
\end{equation}
\end{theorem}
The proof is deferred to Appendix~\ref{app:proof-multilevel-validity}.

Theorem~\ref{thm:multilevel-validity} provides a finite-sample containment guarantee for the future latent conditional law using only observable outcome-level calibration. Neither the historical latent laws nor their mixing kernel needs to be estimated. 
We note that this ambiguity set differs fundamentally from constructing a confidence region around an estimated fixed distribution: the object covered here is itself a future random probability measure.
Also, the guarantee in~\eqref{eq:conditional-law-coverage} is marginal over the calibration data $\mathcal{D}_n$, context $X_{n+1}$, and latent law $\mathbb{Q}_{n+1}$. 

The aggregation rule $\nu$ needs to be fixed or selected using an auxiliary tuning dataset that is independent of the final calibration sample and future observation, and then held fixed during final calibration. This separation preserves Theorem~\ref{thm:multilevel-validity} while allowing the geometry of the ambiguity set to be adapted to the downstream decision problem.

\vspace{-.06in}
\section{Distributionally Robust Framework}
\label{sec:optimization}
\vspace{-.06in}

This section proposes a finite-dimensional shell reformulation of the problem in \eqref{eq:original-dro} under the conformalized ambiguity set. 
Our analysis proceeds in two steps: ($i$) 
We first fix the conformal-level aggregation rule and derive a finite-dimensional representation of the worst-case expectation, together with the structure of its adversarial distribution. 
($ii$) We then study how the aggregation rule itself can be jointly optimized subject to a prescribed complexity budget.

\vspace{-.06in}
\subsection{Evaluating the Worst-Case Expectation}
\label{sec:tractable-dro}
\vspace{-.06in}

We begin by considering aggregation rules with finite support. For an integer $K\geq1$, define
\begin{equation}
\label{eq:finite-nu-class}
\begin{aligned}
    \mathcal P_K(\Gamma)
    \coloneqq &
    \left\{
        \nu
        =
        \sum_{k=1}^K
            w_k\delta_{\gamma_k}:
        \gamma_k\in\Gamma,\;
        w\in\Delta_K
    \right\},\\
    \Delta_K
    \coloneqq &
    \left\{
        w\in\mathbb R_+^K:
        \sum_{k=1}^K w_k=1
    \right\}.
\end{aligned}
\end{equation}
Thus, $\mathcal P_K(\Gamma)$ contains aggregation rules supported on at most
$K$ conformal levels. 

Fix a context $x$, calibration sample $\mathcal D_n$, and aggregation rule $\nu \in \mathcal{P}_K(\Gamma)$, and suppose that $\mathfrak U_\alpha^\nu(x)$ is nonempty. For a bounded measurable loss $c:\mathcal Y\to\mathbb R$, define the robust upper risk
\begin{equation}
\label{eq:conformal-risk}
    \mathcal R_{\alpha,\nu,x}(c(\cdot))
    \coloneqq 
    \sup_{
        \mathbb{Q}\in
        \mathfrak U_\alpha^\nu(x)
    }
    \int_{\mathcal Y}
        c(y)\,
        \mathbb{Q}(\d y),
\end{equation}
which evaluates the largest expected loss
consistent with the conformal evidence encoded by
$\mathfrak U_\alpha^\nu(x)$. From the risk-measure perspective, the ambiguity set serves as the associated risk envelope. Since \eqref{eq:conformal-risk} is a supremum of expectations over a fixed nonempty set of probability measures, it also defines a coherent upper risk functional; the standard verification is deferred to Appendix~\ref{app:coherent-risk}.

The aggregated conformal incompatibility in \eqref{eq:aggregate-score} can be rewritten as
\[
    e_\nu(\mathbb Q;x) =
    \sum_{k=1}^K \frac{w_k}{\gamma_k} \left( 1- \mathbb{Q}(\mathcal{C}_{\gamma_k}(x))\right)
    \coloneqq \int_{\mathcal Y}h_\nu(y;x)\,\mathbb Q(\d y)
\]
where
\begin{equation}
\label{eq:penalty-function}
    h_\nu(y;x)\coloneqq\sum_{k=1}^K\frac{w_k}{\gamma_k}\mathbbm 1 \left\{y\notin\mathcal C_{\gamma_k}(x)\right\},
\end{equation}
Consequently,
\begin{equation}
    \label{eq:reform-conformal-ambiguity-set}
    \mathfrak U_\alpha^\nu(x)
    =
    \left\{
        \mathbb{Q}\in\mathcal P(\mathcal Y):
        \int_{\mathcal Y}
            h_\nu(y;x)
            \mathbb{Q}(\d y)
        \leq
        \frac{1}{\alpha}
    \right\}.
\end{equation}
This representation gives a useful optimization interpretation of the ambiguity set. 
The function $h_\nu(y;x)$ assigns a statistical incompatibility penalty to each outcome $y$, and then the original problem in \eqref{eq:conformal-risk} can be viewed as adversary maximizing expected loss subject to a budget on its expected incompatibility. 
It also admits a Wasserstein representation under a conformal-depth geometry, where the aggregation rule $\nu$ determines the geometry while the desired coverage level fixes the radius; a detailed interpretation is provided in Appendix~\ref{app:wasserstein_interpretation}.
Unlike a conventional moment function specified a priori and used as a distributional constraint in literature, such as \citep{wiesemann2014distributionally}, $h_\nu$ is generated in a data-driven manner from the nested conformal path.

After removing zero-weight and duplicate support points, we may, w.l.o.g., order the remaining conformal miscoverage levels as
$0<\gamma_1<\cdots<\gamma_K<\alpha$. Assume for every context \(x\), the innermost conformal region \(C_{\gamma_K}(x)\) is nonempty.
Because the conformal sets are nested,
$
    \mathcal C_{\gamma_1}(x)
    \supseteq
    \cdots
    \supseteq
    \mathcal C_{\gamma_K}(x).
$
They induce the disjoint \emph{conformal shell} partition
\begin{align*}
    \mathcal{A}_0(x)
    &\coloneqq 
    \mathcal C_{\gamma_K}(x),
    \nonumber\\
    \mathcal{A}_j(x)
    &\coloneqq 
    \mathcal C_{\gamma_{K-j}}(x)
    \setminus
    \mathcal C_{\gamma_{K-j+1}}(x),
    ~
    1 \le j\le K-1,
    \\
    \mathcal{A}_K(x)
    &\coloneqq 
    \mathcal{Y} \setminus C_{\gamma_1}(x).
    \nonumber
\end{align*}
Some shells may be empty, for example, because different conformal levels
produce identical prediction sets. We therefore define
\begin{equation}
\label{eq:nonempty-shell-index}
    \mathcal J(x)
    \coloneqq 
    \left\{
        j\in\{0,\ldots,K\}:
        \mathcal{A}_j(x)\neq\varnothing
    \right\}.
\end{equation}

The incompatibility penalty, denoted by $h_j$, is constant on the outcome $y$ in each nonempty conformal shell $\mathcal{A}_j$. Specifically,
let
\begin{equation}
\label{eq:shell-penalty}
    h_0\coloneqq 0,
    \quad
    h_j
    \coloneqq 
    \sum_{k=K-j+1}^K
        \frac{w_k}{\gamma_k},
    \quad
    j=1,\ldots,K.
\end{equation}
In particular, an outcome in the innermost shell lies within every conformal region and therefore receives zero penalty ($h_0=0$), whereas an outcome in the outermost shell lies outside every conformal region and receives the maximal penalty ($h_K=\sum_{k=1}^K w_k/\gamma_k$); the intermediate penalties $h_j$ capture progressively greater degrees of statistical atypicality, as shown in Figure~\ref{fig:notations}.

For a feasible decision $z$, define the corresponding regional worst-case
loss
\begin{equation}
\label{eq:shell-cost}
    L_j(z;x)
    \coloneqq 
    \sup_{y\in \mathcal{A}_j(x)}
        c(z,y),
    \quad
    j\in\mathcal J(x).
\end{equation}

\begin{assumption}[Finite regional losses]
\label{ass:finite-loss}
For every feasible $(z,x)$ and every $j\in\mathcal J(x)$,
$L_j(z;x)<+\infty$.
\end{assumption}

Assumption~\ref{ass:finite-loss} isolates the regime in which the robust objective is finite. It holds, for example, when $c(z,\cdot)$ is bounded above, or when the relevant outcome regions are compact and $c(z,\cdot)$ is upper semicontinuous. When the loss is unbounded on an outer shell, the conformal-induced robust risk may instead be infinite.

We can now evaluate the inner worst-case expectation.

\begin{theorem}[Worst-case expectation]
\label{thm:shell-reformulation}
Under Assumption~\ref{ass:finite-loss}, for every fixed $(z,x,\nu)$,
\begin{align}
    \mathcal R_{\alpha,\nu,x}
    \bigl(c(z,\cdot)\bigr)
    =
    \max_{\{p_j\}_{j\in\mathcal J(x)}}
    \quad&
    \sum_{j\in\mathcal J(x)}
        p_j L_j(z;x) 
\label{eq:shell-lp}
    \\
    \mathrm{s.t.}\quad&
    p_j\geq0,
    ~
    j\in\mathcal J(x),
\label{eq:shell-lp-const1}
    \\
    &
    \sum_{j\in\mathcal J(x)}
        p_j
    =
    1,
\label{eq:shell-lp-const2}
\\
    &
    \sum_{j\in\mathcal J(x)}
        h_jp_j
    \leq
    \frac{1}{\alpha}.
\label{eq:shell-lp-const3}
\end{align}
Moreover, an optimal shell-mass vector in
\eqref{eq:shell-lp} can be chosen with at most two positive components.
If the regional suprema $L_j(z;x)$ are attained, then there exists a
worst-case probability law supported on at most two outcome points, one in
each active shell. Without attainment, the same statement holds for
$\varepsilon$-optimal laws for every $\varepsilon>0$.
\end{theorem}
The proof is deferred to
Appendix~\ref{app:proof-shell-reformulation}. 

\begin{figure}[!t]
    \centering
    \includegraphics[width=\linewidth]{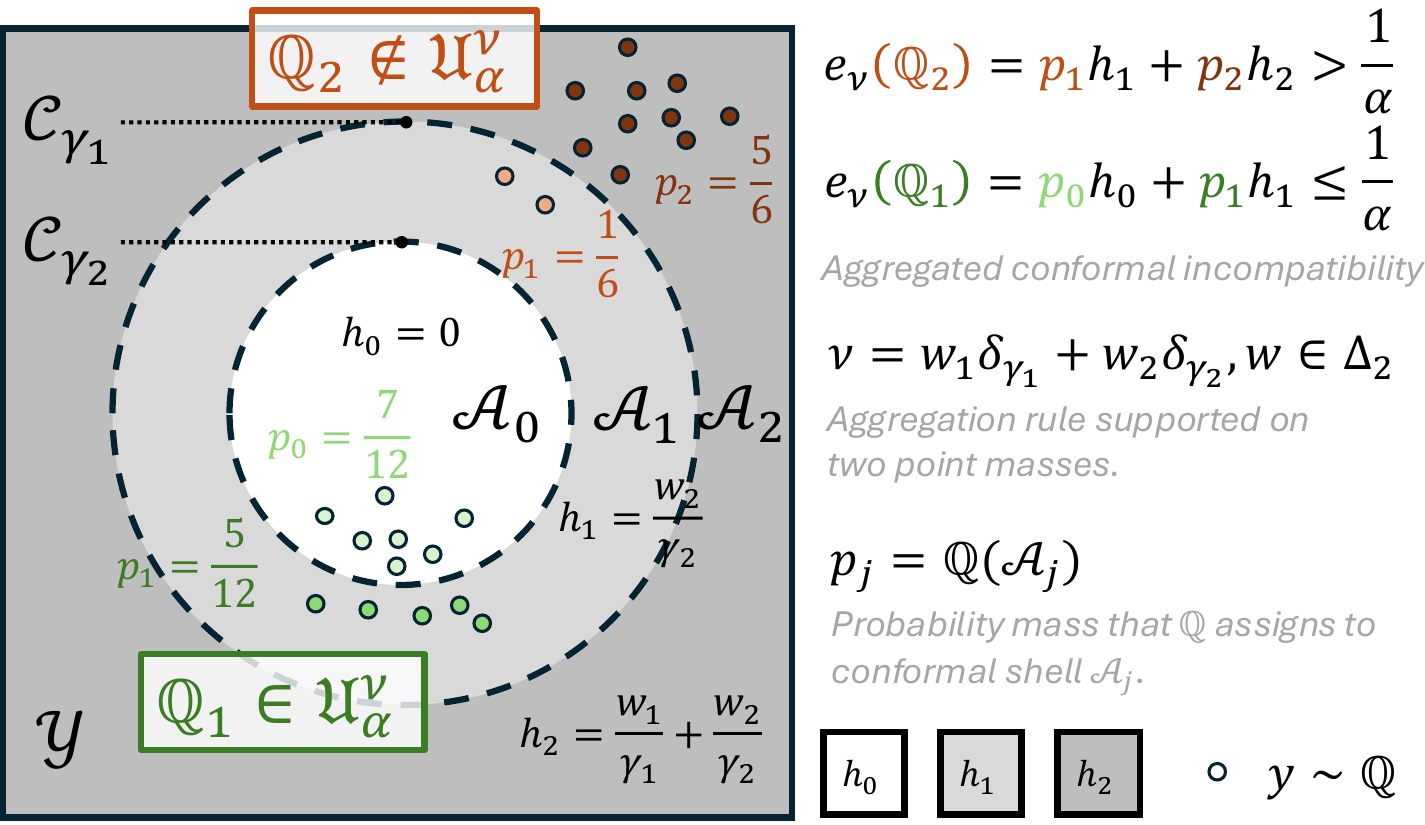}
    \caption{Key notions underlying the ambiguity set with a two-point aggregation rule $\nu=w_1\delta_{\gamma_1}+w_2\delta_{\gamma_2}$. The nested conformal regions partition $\mathcal Y$ into shells $\mathcal A_0,\mathcal A_1,\mathcal A_2$, with penalties $h_0$, $h_1$, and $h_2$. A candidate law $\mathbb Q$, with shell masses $p_j$, belongs to $\mathfrak U_\alpha^\nu$ if $e_\nu(\mathbb Q)\le 1/\alpha$; thus, $\mathbb Q_1$ is included whereas $\mathbb Q_2$ is excluded.
}
    \label{fig:notations}
\end{figure}

Here $p_j\coloneqq \mathbb{Q}(\mathcal{A}_j(x))$ is interpreted as the probability mass that a candidate adversarial distribution $\mathbb{Q}$ assigns to shell $\mathcal{A}_j(x)$. 
The constraints in \eqref{eq:shell-lp-const1} and \eqref{eq:shell-lp-const2} ensure that these shell masses form a probability distribution, while \eqref{eq:shell-lp-const3} limits their average conformal incompatibility according to \eqref{eq:reform-conformal-ambiguity-set}. Intuitively, the adversary allocates probability across shells to emphasize regions with large worst-case losses $L_j(z;x)$ while respecting the statistical-incompatibility budget.

Theorem~\ref{thm:shell-reformulation} separates the distributional optimization from the geometry of the outcome space. Once the regional losses $L_j(z;x)$ have been computed, the adversary no longer optimizes over an infinite-dimensional probability law. It only chooses how much probability mass to allocate across conformal shells subject to one linear incompatibility budget. The two-shell structure follows from this reduction: an extreme point of the shell-mass problem uses at most two active regions.

\vspace{-.06in}
\subsection{Finite-Sample Robust Risk Certificate}
\label{sec:risk-certificate}
\vspace{-.06in}

For the contextual decision problem, define
\begin{equation}
\label{eq:conformal-dro}
    \widehat z_n(x)
    \in
    \argmin_{z\in\mathcal Z(x)}
    \mathcal R_{\alpha,\nu,x}
    \bigl(c(z,\cdot)\bigr),
\end{equation}
and let
\[
    \widehat V_n(x)
    \coloneqq 
    \mathcal R_{\alpha,\nu,x}
    \bigl(
        c(\widehat z_n(x),\cdot)
    \bigr)
\]
denote the corresponding robust value.

Theorem~\ref{thm:shell-reformulation} leads to a direct dual reformulation of the DRO problem defined in \eqref{eq:conformal-dro}:
\begin{equation}
\label{eq:one-dimensional-dual}
    \inf_{\substack{
        z\in\mathcal Z(x),
        \lambda\geq0
    }}
    \left\{
        \frac{\lambda}{\alpha}
        +
        \max_{j\in\mathcal J(x)}
        \left[
            L_j(z;x)
            -
            \lambda h_j
        \right]
    \right\}.
\end{equation}
For a fixed decision $z$, the multi-dimensional maximization over probability measures is therefore reduced to a one-dimensional convex minimization in $\lambda$, after evaluating the regional worst-case losses. 


The resulting decision $\widehat{z}_n(x)$ enjoys the following finite-sample certificate.
\begin{corollary}[Finite-sample decision certificate]
\label{thm:risk-certificate}
Under the conditions of
Theorem~\ref{thm:multilevel-validity}, 
we have
\[
    \mathbb P\!\left\{
        \int_{\mathcal{Y}} c(\widehat z_n(X_{n+1}), y) \,\mathbb{Q}_{n+1}(\d y)
        \leq
        \widehat V_n(X_{n+1})
    \right\}
    \geq
    1-\alpha.
\]
\end{corollary}
The proof is deferred to Appendix~\ref{app:proof-risk-certificate}.

Corollary~\ref{thm:risk-certificate} gives the robust value $\widehat V_n(X_{n+1})$ a direct operational interpretation: It provides a statistically certified upper bound on the expected cost of the selected decision with confidence $1-\alpha$.

\vspace{-.06in}
\subsection{Joint Optimization and Reformulation}
\label{sec:joint-geometry}
\vspace{-.06in}


We jointly optimize the aggregation rule and the decision using an independent tuning set $\mathcal{D}_m \coloneqq (X_i,Y_i)_{i=1}^m$, separate from the calibration set $\mathcal{D}_n$.
\begin{equation}
\widehat V_{m,K}(x)\coloneqq\inf_{z\in\mathcal Z(x),\,\nu\in\mathcal P_K(\Gamma_m)}\mathcal R_{\alpha,\nu,x}\bigl(c(z,\cdot)\bigr).
\label{eq:joint-geometry-design}
\end{equation}
Here, $K$ is a resolution budget that limits the number of conformal levels used to define the tuning geometry. 


For split conformal prediction, the continuum of candidate miscoverage levels reduces to a finite rank grid. Suppose $\alpha>1/(m+1)$. Since
$
    k_\gamma
    =
    \left\lceil (m+1)(1-\gamma)\right\rceil,
$
the conformal region is constant over each interval
$
    \gamma
    \in
    \left[
        r/(m+1),
        (r+1)/(m+1)
    \right),
    ~0 \le r \le m.
$
For $r\geq1$, replacing $\gamma$ by the left endpoint $r/(m+1)$ leaves the conformal region unchanged while weakly increasing the leakage coefficient $1/\gamma$. The resulting ambiguity set can only shrink, so the robust risk cannot increase. The interval corresponding to $r=0$ produces the trivial conformal region $\mathcal Y$ and contributes no incompatibility. Hence, it is sufficient to consider
$
    \Gamma_m
    \coloneqq
    \left\{
        r / (m+1):
        r=1,\ldots,m,\;
        r / (m+1) < \alpha
    \right\}.
$
If multiple levels generate the same conformal region, we retain the smallest such level and continue to denote the reduced grid by $\Gamma_m$.

Apply the shell construction introduced in Section~\ref{sec:tractable-dro} to the full ordered family
$\{\mathcal C_{\gamma}(x): \gamma\in \Gamma_m\}$.
For notational simplicity, we continue to denote the resulting shells by $\mathcal A_j(x)$, their nonempty-shell index set by
$
    \mathcal J(x)
    \coloneqq
    \left\{
        j\in\{0,\ldots,R\}:
        \mathcal A_j(x)\neq\varnothing
    \right\},
$
and their regional worst-case losses by $L_j(z;x)$. Thus, $i=1,\dots,R$ indexes candidate conformal levels, whereas $j\in\mathcal J(x)$ indexes nonempty shells.

\begin{theorem}[Joint optimization reformulation]
\label{thm:cardinality-reformulation}
Suppose $\alpha>1/(m+1)$ and Assumption~\ref{ass:finite-loss} holds for every $j\in\mathcal J(x)$. Then problem~\eqref{eq:joint-geometry-design} is equivalent to
\begin{align*}
    \inf_{\substack{
        z\in\mathcal Z(x),\\
        t\in\mathbb R,\;
        u\in\mathbb R_+^R
    }}
    \quad&
    t+\frac{1}{\alpha}\sum_{i=1}^{R}u_i
    \\
    \mathrm{s.t.}\quad&
    t+
    \sum_{i=R-j+1}^{R}\frac{u_i}{\gamma_i}
    \geq L_j(z;x),
    ~ j\in\mathcal J(x),
    \\
    &
    \|u\|_0\leq K.
\end{align*}
If $(z^\star,t^\star,u^\star)$ is optimal and $\lambda^\star:=\sum_{i=1}^R u_i^\star>0$, then the context-dependent tuning rule is
\[
\widehat\nu_{m,K}(x)=\sum_{i:u_i^\star>0}\widehat w_i(x)\delta_{\gamma_i},\qquad \widehat w_i(x)=\frac{u_i^\star}{\lambda^\star}.
\]
\end{theorem}
The proof is provided in Appendix~\ref{app:proof-thm:cardinality-reformulation}.

The reformulation defines the a measurable map $x\mapsto\widehat\nu_{m,K}(x)$.  It applies the shell representation of Theorem~\ref{thm:shell-reformulation} once to the finest finite-sample conformal partition and uses the support of $u$ to select the active conformal levels. Specifically, $u_i>0$ indicates that level $\gamma_i$ is selected, while the normalized magnitude $u_i/\sum_{\ell=1}^{R}u_\ell$ determines its aggregation weight. Hence, $K$ controls the maximum resolution of the ambiguity geometry without requiring explicit enumeration of candidate support sets.

The statistical interpretation of~\eqref{eq:joint-geometry-design} requires care. Theorem~\ref{thm:multilevel-validity} and Corollary~\ref{thm:risk-certificate} apply directly only when the aggregation rule is independent of both the final calibration set and the future observation. A pointwise solution of~\eqref{eq:joint-geometry-design} instead defines a context-dependent map $x\mapsto\widehat\nu_{m,K}(x)$, whose evaluation at $X_{n+1}$ remains dependent on the future context even though the map is learned using only $\mathcal D_m$. To restore validity, we freeze this map and the tuning-sample conformal regions, regard the resulting context-dependent penalty $h_{\widehat\nu_{m,K}(x)}(y;x)$ as a raw nonconformity score, and use $\mathcal D_n$ to conformally normalize it into a split conformal $e$-predictor 
\citep{vovk2025conformal}, which serves as the final incompatibility penalty. 
Details are provided in Appendix~\ref{sec:proposed-algorithm}.


\vspace{-.06in}
\section{Experiments}
\vspace{-.06in}

We evaluate \texttt{Conformal-DRO} through numerical experiments that examine its learned ambiguity geometry and compare its decision performance and coverage with benchmarks under latent-law heterogeneity.


\vspace{-.06in}
\subsection{Synthetic Examples}
\label{sec:synthetic-experimental-setup}
\vspace{-.06in}


We first consider a covariate-free scalar setting with either a centered unit-scale Gaussian or Student-$t_3$ outcome distribution and linear cost $c(z,y)=z^\top y$. This design isolates the learned ambiguity geometry from prediction error and makes the resulting shells directly visualizable. Implementation details are provided in Appendix~\ref{app:synthetic-experimental-details}.
For \texttt{Conformal-DRO}, an auxiliary sample of size $500$ selects the aggregation rule $\nu_K=\sum_{k=1}^K w_k\delta_{\gamma_k}$, which is then held fixed while an independent sample of size $500$ calibrates the conformal regions. We set $\alpha=0.2$ and compare against a population oracle that uses exact equal-tailed regions of the corresponding truncated population law. 

The shell-allocation comparison in Figure~\ref{fig:pop-syn-linear-shells} shows that \texttt{Conformal-DRO} recovers the conformal levels identified by the population-optimal geometry. Increasing $K$ refines the shell partition and spreads weight across more quantile boundaries, primarily in central and moderately extreme regions. The Student-$t_3$ allocations in Figure~\ref{fig:add-pop-syn-linear-shells} (Appendix~\ref{app:synthetic-experimental-details}) extend farther into the tail because of their slower decay and wider quantile spacing. Although finite-sample calibration shifts individual boundaries and weights, it preserves their nested ordering and overall profile. Thus, $K$ controls geometric resolution, while the outcome distribution determines the locations and weights of the selected shells.

Figure~\ref{fig:optimal-value} shows sharp improvements for small $K$ followed by diminishing returns. The population-oracle values decrease smoothly, whereas the finite-sample values exhibit minor nonmonotonicity because tuning and calibration use independent samples. At $K=100$, the remaining gaps are approximately $0.054$ for the Gaussian distribution and $0.176$ for Student-$t_3$, whose heavier tails make extreme quantiles more variable and operationally important. The resulting plateau indicates that, once discretization error is small, the remaining gap is driven by finite-sample tuning and calibration and cannot be removed by increasing $K$ alone.

\begin{figure}[!t]
    \centering
    \begin{subfigure}[t]{\linewidth}
        \centering
        \includegraphics[width=.32\linewidth]{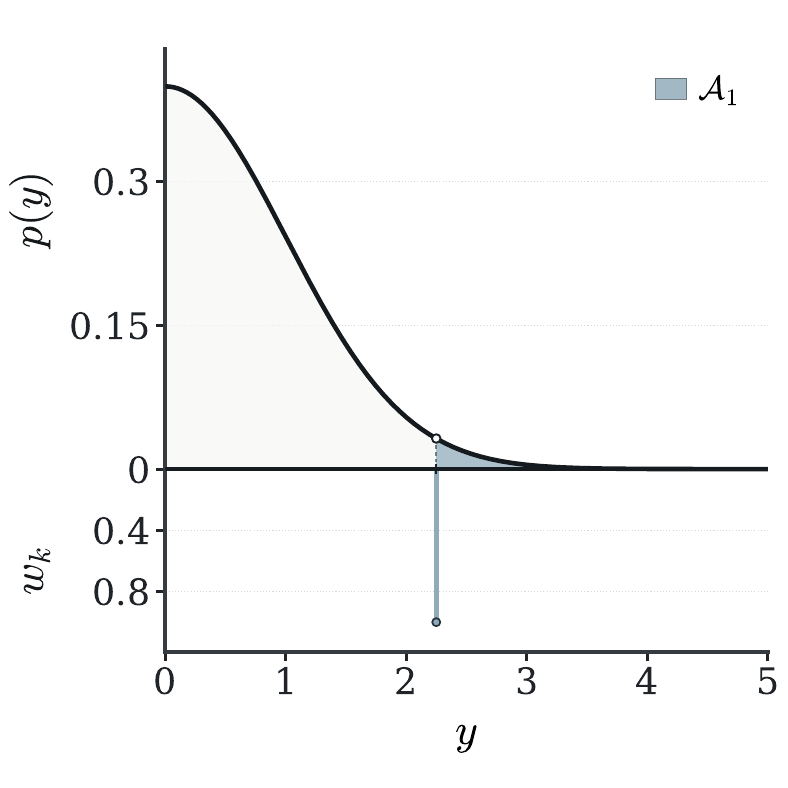}
        \includegraphics[width=.32\linewidth]{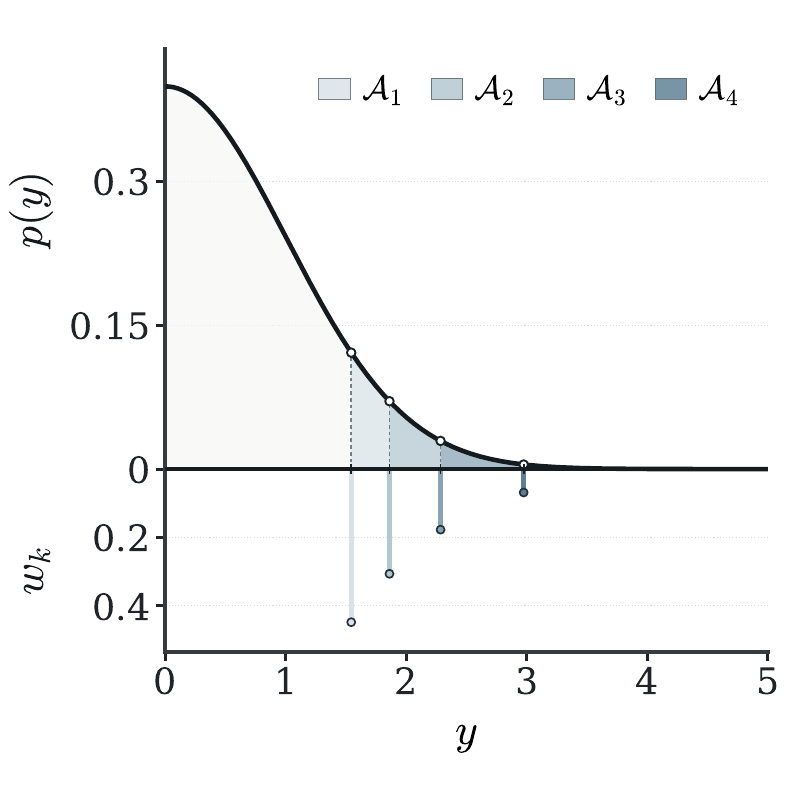}
        \includegraphics[width=.32\linewidth]{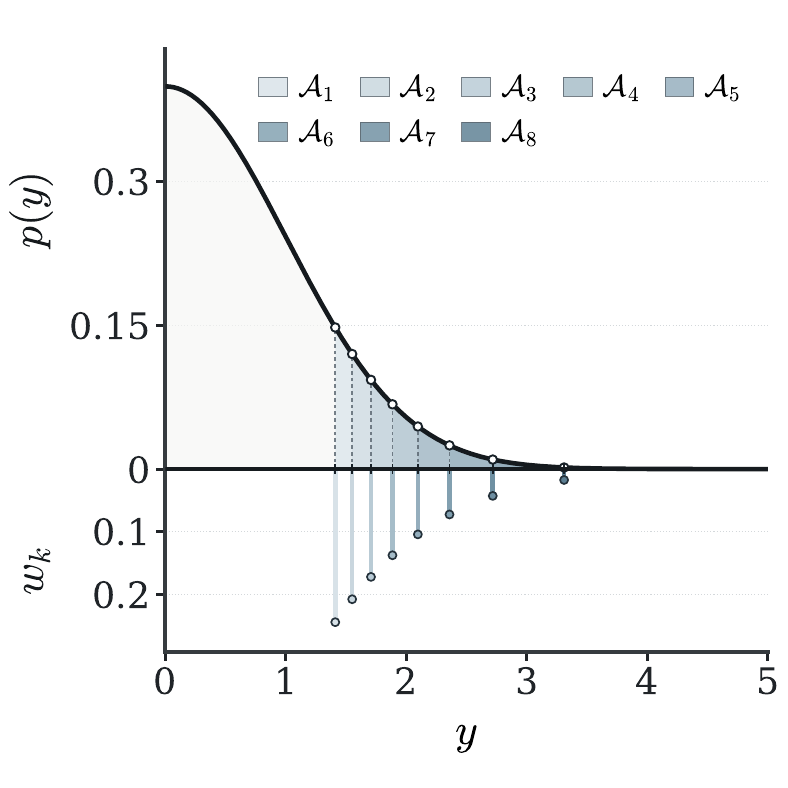}
        \caption{Population-oracle}
    \end{subfigure}
    \vfill
    \begin{subfigure}[t]{\linewidth}
        \centering
        \includegraphics[width=.32\linewidth]{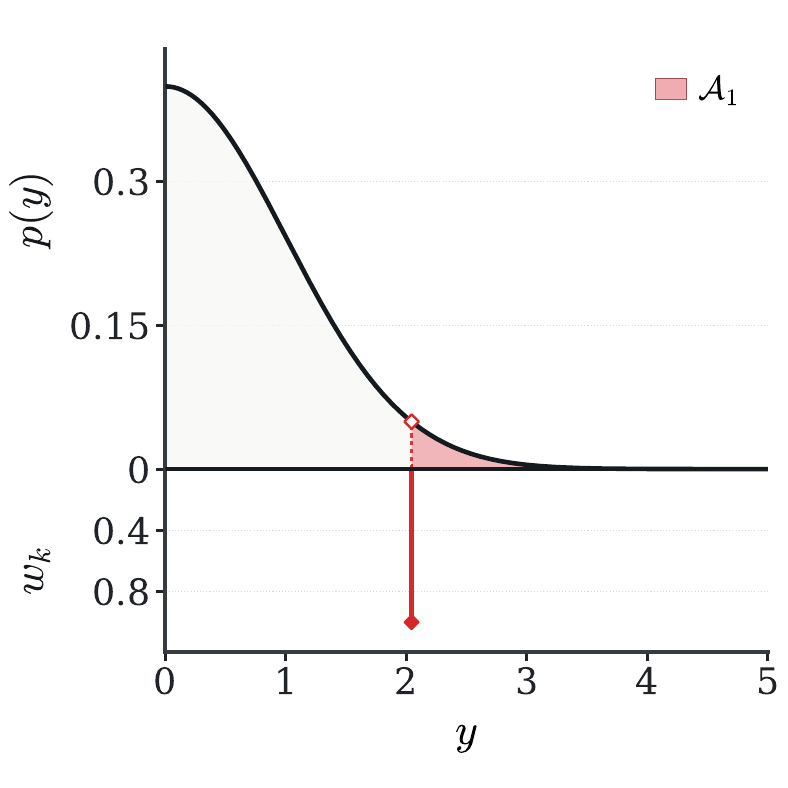}
        \includegraphics[width=.32\linewidth]{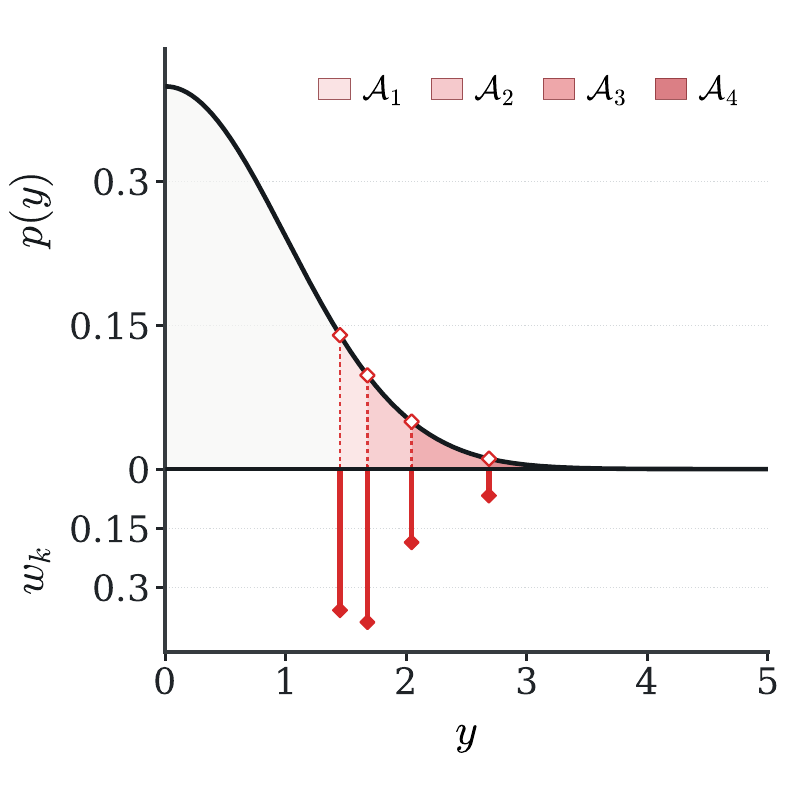}
        \includegraphics[width=.32\linewidth]{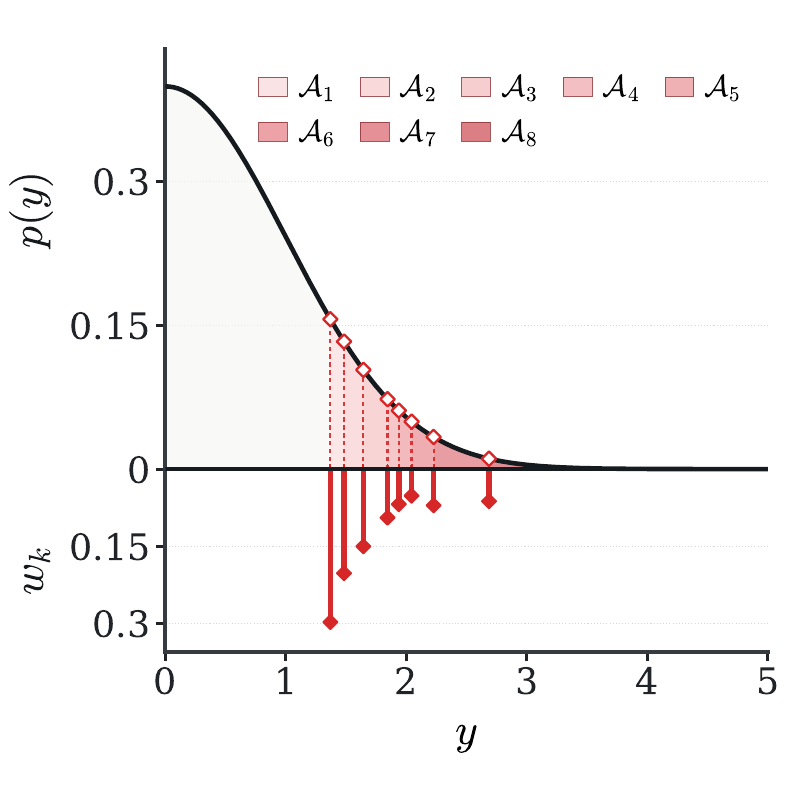}
        \caption{\texttt{Conformal-DRO}}
    \end{subfigure}
    \vspace{-.1in}
    \caption{
    Population-oracle and \texttt{Conformal-DRO} shell allocations under the linear cost $c(z,y)=y^\top z$. 
    Black curves show the population densities, shaded regions the induced shells, dashed lines the optimized upper quantile endpoints, and downward stems the aggregation weights $w_k$. Only $y\geq0$ is shown; $\alpha=0.2$.
    }
    \label{fig:pop-syn-linear-shells}
    \vspace{-.15in}
\end{figure}

\begin{figure}[!t]
    \vspace{-.1in}
    \centering
    \begin{subfigure}[t]{0.49\linewidth}
        \centering
        \includegraphics[width=\linewidth]{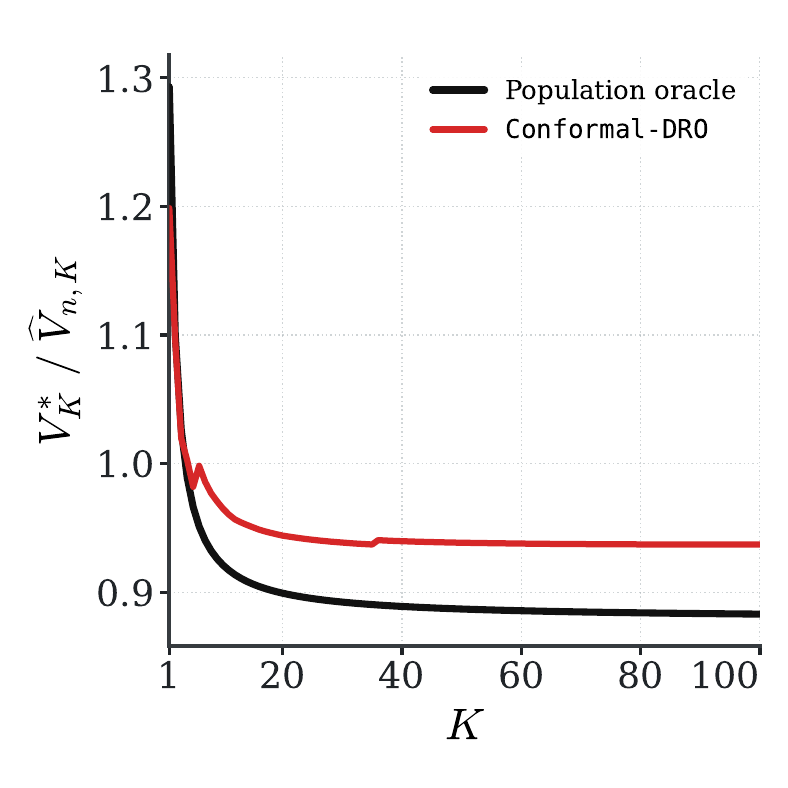}
        \caption{Gaussian}
    \end{subfigure}
    \hfill
    \begin{subfigure}[t]{0.49\linewidth}
        \centering
        \includegraphics[width=\linewidth]{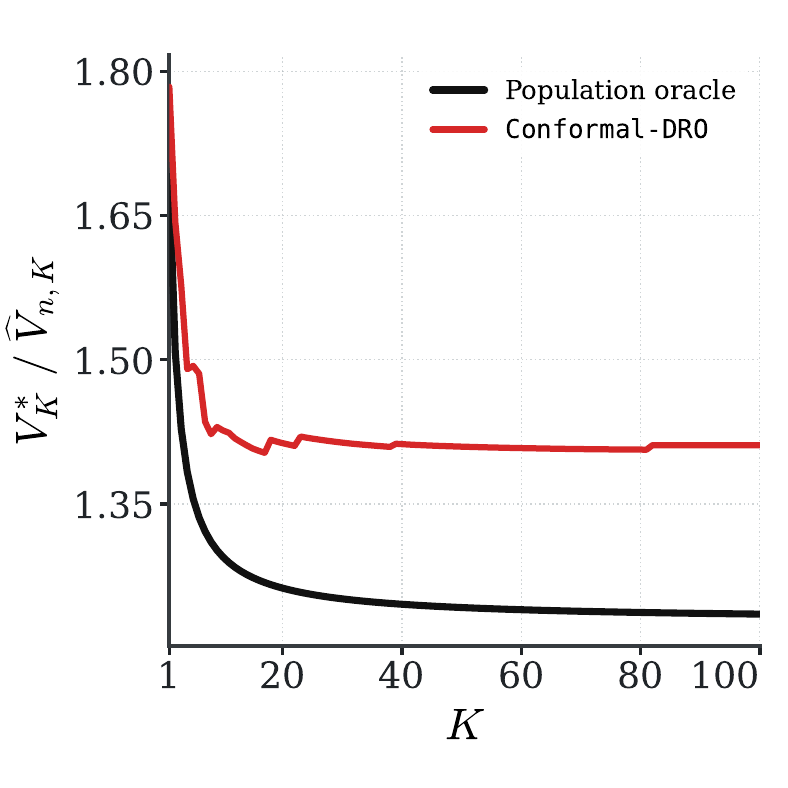}
        \caption{Student-$t_3$}
    \end{subfigure}
    \vspace{-.1in}
    \caption{Population-oracle values $V_K^\star$ and finite-sample \texttt{Conformal-DRO} values $\widehat V_{n,K}$ as functions of the support budget $K$. 
    }
    \label{fig:optimal-value}
    \vspace{-.1in}
\end{figure}

\vspace{-.06in}
\subsection{Comparison with Baselines}
\vspace{-.06in}
\label{subsec:baseline-comparison}

We compare \texttt{Conformal-DRO} with a mixture stochastic program (Mixture SP) \citep{bertsimas2020predictive}, mixture-centered Wasserstein DRO (Mixture WDRO) \citep{zhang2024optimal} in the capacitated two-product newsvendor problem detailed in Appendix~\ref{app:baseline-comparison}. An omitted event indicator $W$ determines the realized law $\mathbb{Q}_{x,W}$, whereas the observable conditional law is the mixture $\mathbb{Q}_x=(1-p(x))\mathbb{Q}_{x,0}+p(x)\mathbb{Q}_{x,1}$. To isolate this target mismatch from estimation error, the exact $\mathbb{Q}_x$ is supplied to both mixture-centered baselines and the conformal score.

\begin{table}[t]
\centering
\caption{Baseline comparison under strong latent-law heterogeneity. Entries are means over ten repetitions.}
\label{tab:baseline-comparison}
\vspace{-.1in}
\footnotesize
\setlength{\tabcolsep}{3pt}
\renewcommand{\arraystretch}{0.95}
\begin{tabular}{@{}lccc@{}}
\toprule
Method & Regret $\downarrow$ & Law cov. $\uparrow$ & Cert. cov. $\uparrow$ \\
\midrule
Mixture SP & $0.848$ & -- & $0.804$ \\
Mixture WDRO & $0.853$ & $0.345$ & $0.849$ \\
\texttt{Conformal-DRO}, $K=1$ & $1.616$ & $0.927$ & $1.000$ \\
\texttt{Conformal-DRO}, $K=4$ & $1.338$ & $0.946$ & $1.000$ \\
\texttt{Conformal-DRO}, $K=8$ & $1.319$ & $0.945$ & $1.000$\\
\bottomrule
\end{tabular}
\end{table}

Table~\ref{tab:baseline-comparison} shows a clear efficiency--protection tradeoff. The mixture methods attain lower mean regret by optimizing around the exact mixture law, but provide substantially weaker protection for the future realized law: mixture WDRO covers only $0.345$ of realized laws, while mixture SP does not construct a law-level ambiguity set. In contrast, all \texttt{Conformal-DRO} variants exceed the nominal law-coverage target of $1-\alpha=0.8$ and attain certificate coverage of $1.000$. Increasing $K$ improves efficiency without materially changing coverage: moving from $K=1$ to $K=8$ reduces regret from $1.616$ to $1.319$, a $18.4\%$ reduction, while maintaining law coverage near $0.95$. Thus, additional conformal levels mitigate the conservatism of a single contour while preserving reliable protection against latent-law heterogeneity.


\vspace{-.06in}
\section{Conclusion}
\vspace{-.06in}

We introduced \texttt{Conformal-DRO} for decision-making under latent distributional heterogeneity. Outcome-level conformal validity yields finite-sample coverage of the future latent law. The corresponding conformal path determines a data-driven transport geometry, while the prescribed miscoverage level calibrates the radius. Its shell structure reduces the infinite-dimensional DRO problem to a finite-dimensional optimization, making latent-law uncertainty both statistically valid and operationally tractable.


\bibliographystyle{plainnat}
\bibliography{refs}

\clearpage
\appendix
\thispagestyle{empty}

\onecolumn
\aistatstitle{
Supplementary Materials 
}

\section{Proof of Theorem~\ref{thm:multilevel-validity}}
\label{app:proof-multilevel-validity}

\begin{proof}
For every fixed $\gamma\in\Gamma$, Eq.~\eqref{eq:e-moment} gives
\[
    \mathbb E\!\left[
        q_\gamma\!\left(
            \mathbb Q_{n+1};X_{n+1}
        \right)
    \right]
    \leq 1.
\]
Since the integrand is nonnegative, Tonelli's theorem yields
\[
    \mathbb E\!\left[
        e_\nu\!\left(
            \mathbb Q_{n+1};X_{n+1}
        \right)
    \right]
    =
    \int_\Gamma
        \mathbb E\!\left[
            q_\gamma\!\left(
                \mathbb Q_{n+1};X_{n+1}
            \right)
        \right]
        \nu(\d\gamma)
    \leq
    \int_\Gamma 1\,\nu(\d\gamma)
    =1.
\]
Markov's inequality therefore implies
\[
    \mathbb P\!\left\{
        e_\nu\!\left(
            \mathbb Q_{n+1};X_{n+1}
        \right)
        >
        \frac{1}{\alpha}
    \right\}
    \leq\alpha,
\]
which is equivalent to~\eqref{eq:conditional-law-coverage}.
\end{proof}

\section{Coherent-Risk Properties}
\label{app:coherent-risk}

Fix $(x,\mathcal D_n,\nu)$ and suppose $\mathfrak U_\alpha^\nu(x)$ is nonempty. Recall
\[
    \mathcal R_{\alpha,\nu,x}(c)
    =
    \sup_{
        \mathbb{Q}\in
        \mathfrak U_\alpha^\nu(x)
    }
    \mathbb E_\mathbb{Q}[c].
\]

\begin{proposition}[Coherence of the conformal upper risk]
\label{prop:coherence}
For bounded measurable losses $c,c_1,c_2$, constants $a\in\mathbb R$, and $\lambda\geq0$,
$\mathcal R_{\alpha,\nu,x}$ satisfies:
\begin{enumerate}
    \item monotonicity:
    if $c_1\leq c_2$, then
    \[
        \mathcal R_{\alpha,\nu,x}(c_1)
        \leq
        \mathcal R_{\alpha,\nu,x}(c_2);
    \]

    \item translation equivariance:
    \[
        \mathcal R_{\alpha,\nu,x}(c+a)
        =
        \mathcal R_{\alpha,\nu,x}(c)+a;
    \]

    \item positive homogeneity:
    \[
        \mathcal R_{\alpha,\nu,x}(\lambda c)
        =
        \lambda
        \mathcal R_{\alpha,\nu,x}(c);
    \]

    \item subadditivity:
    \[
        \mathcal R_{\alpha,\nu,x}(c_1+c_2)
        \leq
        \mathcal R_{\alpha,\nu,x}(c_1)
        +
        \mathcal R_{\alpha,\nu,x}(c_2).
    \]
\end{enumerate}
Therefore $\mathcal R_{\alpha,\nu,x}$ is a coherent upper risk functional.
\end{proposition}

\begin{proof}
Each property follows from the representation as a supremum of expectations over a fixed nonempty set of probability measures.

For monotonicity, $c_1\leq c_2$ implies $\mathbb E_\mathbb{Q}[c_1]\leq\mathbb E_\mathbb{Q}[c_2]$ for every feasible $\mathbb{Q}$.

For translation equivariance,
\[
    \mathcal R_{\alpha,\nu,x}(c+a) =
    \sup_\mathbb{Q}
    \left(
        \mathbb E_\mathbb{Q}[c]+a
    \right) =
    \mathcal R_{\alpha,\nu,x}(c)+a.
\]

For $\lambda\geq0$,
\[
    \mathcal R_{\alpha,\nu,x}(\lambda c)
    =
    \sup_\mathbb{Q}
        \lambda\mathbb E_\mathbb{Q}[c]
    =
    \lambda
    \mathcal R_{\alpha,\nu,x}(c).
\]

Finally,
\[
    \mathcal R_{\alpha,\nu,x}(c_1+c_2)
    =
    \sup_\mathbb{Q}
    \left\{
        \mathbb E_\mathbb{Q}[c_1]
        +
        \mathbb E_\mathbb{Q}[c_2]
    \right\} \leq
    \sup_\mathbb{Q}\mathbb E_\mathbb{Q}[c_1]
    +
    \sup_\mathbb{Q}\mathbb E_\mathbb{Q}[c_2],
\]
which proves subadditivity.
\end{proof}

\section{Geometric Interpretation of Conformalized Ambiguity Set}
\label{app:wasserstein_interpretation}

Fix a context $x$, calibration sample $\mathcal{D}_n$, and aggregation rule $\nu$. 
Assuming that $h_\nu(\cdot;x)$ is finite-valued, define the conformal-depth pseudometric
\begin{equation}
    d_{\nu,x}(y,y')
    \coloneqq
    \left|h_\nu(y;x)-h_\nu(y';x)\right|.
    \label{eq:conformal-depth-pseudometric}
\end{equation}
This geometry measures differences in statistical depth rather than physical distance: distinct outcomes have zero distance whenever they receive the same incompatibility penalty. Consequently, $d_{\nu,x}$ induces a genuine metric on the quotient space obtained by identifying outcomes with the same conformal depth.

Let
$
    \mathcal Y_{\nu,x}^0
    \coloneqq
    \left\{
        y\in\mathcal Y:
        h_\nu(y;x)=0
    \right\}
$
be nonempty, and fix any $y_0\in\mathcal Y_{\nu,x}^0$. Since $d_{\nu,x}(y_0,y)=h_\nu(y;x)$ and the only coupling between $\delta_{y_0}$ and $\mathbb Q$ is $\delta_{y_0}\otimes\mathbb Q$, we obtain
\[
    W_{1,d_{\nu,x}}(\delta_{y_0},\mathbb Q)
    =
    \int_{\mathcal Y}h_\nu(y;x)\,\mathbb Q(dy)
    =
    e_\nu(\mathbb Q;x).
\]
Consequently, the conformal-induced ambiguity set admits the Wasserstein-like representation
\[
    \mathfrak U_\alpha^\nu(x)
    =
    \left\{
        \mathbb Q\in\mathcal P(\mathcal Y):
        W_{1,d_{\nu,x}}(\delta_{y_0},\mathbb Q)
        \leq
        \frac{1}{\alpha}
    \right\},
\]
The source measure $\delta_{y_0}$ is only an auxiliary zero-depth anchor, not a nominal estimate of the future conditional law; the substantive object is the learned pseudometric $d_{\nu,x}$.

More precisely, we can define the conformal-depth map
$
    T_{\nu,x}(y)
    \coloneqq
    h_\nu(y;x).
$
For any probability measures $\mathbb Q_1$ and $\mathbb Q_2$ with finite first $h_\nu$-moments,
\[
    W_{1,d_{\nu,x}}(\mathbb Q_1,\mathbb Q_2)
    =
    W_1\!\left(
        (T_{\nu,x})_\#\mathbb Q_1,
        (T_{\nu,x})_\#\mathbb Q_2
    \right).
\]
Thus, the learned geometry can also be interpreted as the one-dimensional Wasserstein geometry of conformal depth and deliberately discards distinctions that are invisible to the conformal path. In particular, two probability laws have zero conformal-depth distance whenever they induce the same distribution of $h_\nu(Y;x)$, even if the original laws differ.

A key implication of this geometric representation is that the Wasserstein radius $1/\alpha$ is not an arbitrary robustness parameter but has a direct statistical interpretation: $\alpha$ is the desired miscoverage level for the future latent conditional law, so the resulting ambiguity set contains that law with probability at least $1-\alpha$. Thus, selecting a smaller $\alpha$ simultaneously requests greater confidence and enlarges the corresponding conformal-depth Wasserstein ball. Unlike conventional Wasserstein DRO, where the radius is typically calibrated through empirical tuning or concentration bounds, the radius in our construction is determined directly by the desired coverage guarantee. 
The aggregation rule $\nu$ still determines the geometry and tightness of the ambiguity set, but the radius itself is fixed by the decision maker’s prescribed confidence level.

\section{Proof of Theorem~\ref{thm:shell-reformulation}}
\label{app:proof-shell-reformulation}

\begin{proof}
Fix $x$ and define the index set of nonempty shells
\[
\mathcal J(x)
=
\{j\in\{0,\ldots,K\}:\mathcal{A}_j(x)\neq\varnothing\}.
\]
Only these shells are retained below. 
Fix $z\in\mathcal Z(x)$, and suppress the dependence on $(x,\mathcal D_n,z)$ when no confusion can arise. By construction, $\mathcal{A}_0(x),\ldots,\mathcal{A}_K(x)$ form a partition of $\mathcal Y$, and
\[
    h_\nu(y;x)=h_j
    \qquad
    \text{for every }y\in \mathcal{A}_j(x).
\]

\paragraph{Step 1: Exact reduction to shell probabilities.}
Every probability measure $\mathbb Q\in\mathcal P(\mathcal Y)$ can be represented as
\[
    \mathbb Q
    =
    \sum_{j=0}^Kp_j\mathbb Q_j,
\]
where
\[
    p_j=\mathbb Q(\mathcal{A}_j(x))
\]
and $\mathbb Q_j$ is supported on $\mathcal{A}_j(x)$. When $p_j>0$, one may take
\[
    \mathbb Q_j(\mathcal{B})
    =
    \frac{\mathbb Q(\mathcal{B}\cap \mathcal{A}_j(x))}{p_j}
\]
for every measurable $\mathcal{B}\subseteq\mathcal Y$. When $p_j=0$, the choice of $\mathbb Q_j$ is immaterial. Conversely, every such mixture defines a probability measure on $\mathcal Y$.

Since $h_\nu=h_j$ on $\mathcal{A}_j(x)$, the ambiguity-set constraint satisfies
\[
    \int_{\mathcal Y}
        h_\nu(y;x)\,
        \mathbb Q(\mathrm{d}y)
    =
    \sum_{j=0}^Kh_jp_j.
\]
The conformal upper risk can therefore be expressed as
\begin{equation}
\label{eq:proof-mixture-formulation}
\begin{aligned}
    \mathcal R_{\alpha,\nu,x}\bigl(c(z,\cdot)\bigr)
    =
    \sup_{p,\mathbb Q_0,\ldots,\mathbb Q_K}
    \quad&
    \sum_{j=0}^K
        p_j
        \int_{\mathcal Y}
            c(z,y)\,
            \mathbb Q_j(\mathrm{d}y)
    \\
    \mathrm{s.t.}\quad&
    p\in\mathbb R_+^{K+1},
    \\
    &
    \sum_{j=0}^Kp_j=1,
    \\
    &
    \sum_{j=0}^Kh_jp_j\leq\frac{1}{\alpha},
    \\
    &
    \mathbb Q_j\in\mathcal P(\mathcal Y),
    \qquad j \in \mathcal J(x),
    \\
    &
    \mathbb Q_j(\mathcal{A}_j(x))=1,
    \qquad j \in \mathcal J(x).
\end{aligned}
\end{equation}

For fixed $p$, the measures $\mathbb Q_0,\ldots,\mathbb Q_K$ can be optimized independently. For each shell, consider
\begin{equation}
\label{eq:proof-within-shell-problem}
\begin{aligned}
    \sup_{\mathbb Q_j}
    \quad&
    \int_{\mathcal Y}
        c(z,y)\,
        \mathbb Q_j(\mathrm{d}y)
    \\
    \mathrm{s.t.}\quad&
    \mathbb Q_j\in\mathcal P(\mathcal Y),
    \\
    &
    \mathbb Q_j(\mathcal{A}_j(x))=1.
\end{aligned}
\end{equation}
The optimal value of \eqref{eq:proof-within-shell-problem} is
\[
    \sup_{y\in \mathcal{A}_j(x)}c(z,y)
    =
    L_j(z;x).
\]
Indeed, no probability measure supported on $\mathcal{A}_j(x)$ can have an expected cost larger than the pointwise supremum over that shell, while point masses at maximizers or approximate maximizers attain or approach this value.

Substituting these shell-wise optimal values into \eqref{eq:proof-mixture-formulation} gives
\begin{equation}
\label{eq:proof-shell-primal}
\begin{aligned}
\mathcal R_{\alpha,\nu,x}(c(z,\cdot))
=
\max_{\{p_j\}_{j\in\mathcal J(x)}}
\quad&
\sum_{j\in\mathcal J(x)}p_jL_j(z;x)
\\
\mathrm{s.t.}\quad&
p_j\geq0,
\qquad j\in\mathcal J(x),
\\
&
\sum_{j\in\mathcal J(x)}p_j=1,
\\
&
\sum_{j\in\mathcal J(x)}h_jp_j
\leq\frac1\alpha.
\end{aligned}
\end{equation}
Under Assumption~\ref{ass:finite-loss}, all $L_j(z;x)$ are finite. The feasible set is nonempty because $h_0=0$ and $p_0=1$ is feasible. It is also compact as a closed subset of the probability simplex. Thus, the supremum is attained, and \eqref{eq:proof-shell-primal} proves \eqref{eq:shell-lp}.

\paragraph{Step 2: Dual reformulation.}
Introduce a free dual variable $t\in\mathbb R$ for the normalization constraint and a nonnegative dual variable $\lambda\geq0$ for the incompatibility constraint. The Lagrangian of \eqref{eq:proof-shell-primal} is
\begin{align}
    \mathcal L(p,t,\lambda)
    &=
    \sum_{j=0}^Kp_jL_j(z;x)
    +
    t\left(1-\sum_{j=0}^Kp_j\right)
    +
    \lambda
    \left(
        \frac{1}{\alpha}
        -
        \sum_{j=0}^Kh_jp_j
    \right)
    \nonumber\\
    &=
    t+\frac{\lambda}{\alpha}
    +
    \sum_{j=0}^K
        p_j
        \bigl[
            L_j(z;x)-t-\lambda h_j
        \bigr].
    \label{eq:proof-shell-lagrangian}
\end{align}
Maximizing the Lagrangian over $p\in\mathbb R_+^{K+1}$ gives a finite value if and only if
\[
    t+\lambda h_j
    \geq
    L_j(z;x),
    \qquad
    j \in \mathcal J(x).
\]
The dual problem is therefore
\begin{equation}
\label{eq:proof-shell-dual}
\begin{aligned}
    \inf_{t,\lambda}
    \quad&
    t+\frac{\lambda}{\alpha}
    \\
    \mathrm{s.t.}\quad&
    t\in\mathbb R,
    \\
    &
    \lambda\geq0,
    \\
    &
    t+\lambda h_j
    \geq
    L_j(z;x),
    \qquad j \in \mathcal J(x).
\end{aligned}
\end{equation}
The primal problem is feasible and bounded. Strong linear-programming duality therefore implies that \eqref{eq:proof-shell-primal} and \eqref{eq:proof-shell-dual} have the same optimal value.

For any fixed $\lambda\geq0$, the smallest feasible value of $t$ is
\[
    t(\lambda)
    =
    \max_{j=0,\ldots,K}
    \bigl[
        L_j(z;x)-\lambda h_j
    \bigr].
\]
Eliminating $t$ from \eqref{eq:proof-shell-dual} yields
\begin{equation}
\label{eq:proof-one-dimensional-dual}
\begin{aligned}
    \mathcal R_{\alpha,\nu,x}\bigl(c(z,\cdot)\bigr)
    =
    \inf_{\lambda}
    \quad&
    \frac{\lambda}{\alpha}
    +
    \max_{j=0,\ldots,K}
    \bigl[
        L_j(z;x)-\lambda h_j
    \bigr]
    \\
    \mathrm{s.t.}\quad&
    \lambda\geq0,
\end{aligned}
\end{equation}
which proves \eqref{eq:one-dimensional-dual}.

\paragraph{Step 3: Sparsity of a worst-case law.}
Introduce a slack variable $s\geq0$ and write the shell linear program in standard form:
\begin{equation}
\label{eq:proof-shell-standard-form}
\begin{aligned}
    \max_{p,s}
    \quad&
    \sum_{j=0}^Kp_jL_j(z;x)
    \\
    \mathrm{s.t.}\quad&
    \sum_{j=0}^Kp_j=1,
    \\
    &
    \sum_{j=0}^Kh_jp_j+s=\frac{1}{\alpha},
    \\
    &
    p\in\mathbb R_+^{K+1},
    \\
    &
    s\geq0.
\end{aligned}
\end{equation}
The standard-form problem has at most two linearly independent equality constraints. By the fundamental theorem of linear programming, there exists an optimal basic feasible solution $(p^\star,s^\star)$ satisfying
\[
    \bigl|\{j:p_j^\star>0\}\bigr|
    \leq2.
\]

If the regional suprema are attained, choose
\[
    y_j^\star
    \in
    \argmax_{y\in \mathcal{A}_j(x)}c(z,y)
\]
for every $j$ such that $p_j^\star>0$, and define
\[
    \mathbb Q^\star
    =
    \sum_{j:p_j^\star>0}
        p_j^\star\delta_{y_j^\star}.
\]
Then
\[
    \int_{\mathcal Y}
        h_\nu(y;x)\,
        \mathbb Q^\star(\mathrm{d}y)
    =
    \sum_{j=0}^Kh_jp_j^\star
    \leq
    \frac{1}{\alpha},
\]
and
\[
    \int_{\mathcal Y}
        c(z,y)\,
        \mathbb Q^\star(\mathrm{d}y)
    =
    \sum_{j=0}^Kp_j^\star L_j(z;x)
    =
    \mathcal R_{\alpha,\nu,x}\bigl(c(z,\cdot)\bigr).
\]
Thus, an optimal worst-case law can be chosen with support in at most two shells. If the regional suprema are not attained, approximate maximizers yield an $\varepsilon$-optimal law supported on at most two shells for every $\varepsilon>0$.
\end{proof}

\section{Proof of Corollary~\ref{thm:risk-certificate}}
\label{app:proof-risk-certificate}

\begin{proof}
Let
\[
E\coloneqq \left\{
\mathbb{Q}_{n+1}
\in
\mathcal U_{\alpha}^{\nu}(X_{n+1})
\right\}.
\]
By Theorem~1, $\mathbb P(E)\ge 1-\alpha$. On the event $E$, for every
$z\in\mathcal Z(X_{n+1})$,
\[
\int_{\mathcal Y} c(z,y)\,
\mathbb{Q}_{n+1}(\d y\mid X_{n+1})
\le
\sup_{\mathbb{Q}\in\mathcal U_{\alpha}^{\nu}(X_{n+1})}
\int_{\mathcal Y} c(z,y)\,\mathbb{Q}(\d y)
=
\mathcal R_{\alpha,\nu,X_{n+1}}\bigl(c(z,\cdot)\bigr).
\]
Thus, the inequality holds simultaneously for all feasible $z$ with probability at least $1-\alpha$.

In particular, substituting the data-dependent decision $z=\widehat z_n(X_{n+1})$ and using the definition of $\widehat V_n(X_{n+1})$ gives
\[
\int_{\mathcal Y}
c\bigl(\widehat z_n(X_{n+1}),y\bigr)
\mathbb{Q}_{n+1}(\d y)
\le
\widehat V_n(X_{n+1})
\]
on $E$. Since $\mathbb P(E)\ge 1-\alpha$, the claimed probability guarantee follows.
\end{proof}

\section{Proof of Theorem~\ref{thm:cardinality-reformulation}}
\label{app:proof-thm:cardinality-reformulation}

\begin{proof}
The proof proceeds in four steps.

\paragraph{Step 1: Reduction to the finite conformal-rank grid.}
For any $\gamma\in(0,\alpha)$, there exists a unique
$r\in\{0,\ldots,m\}$ such that
\[
    \gamma\in
    \left[
        \frac{r}{m+1},
        \frac{r+1}{m+1}
    \right).
\]
It follows that
\[
    (m+1)(1-\gamma)\in(n-r,m+1-r]
\]
and hence
\[
    k_\gamma
    =
    \left\lceil (m+1)(1-\gamma)\right\rceil
    =
    m+1-r.
\]
Therefore, the conformal quantile and prediction region are constant
throughout this interval.

Suppose first that $r\geq1$ and let
\[
    \widetilde\gamma=\frac{r}{m+1}.
\]
Because $\widetilde\gamma\leq\gamma$ and
$
    \mathcal C_{\widetilde\gamma}(x)
    =
    \mathcal C_\gamma(x),
$
we have, for every $y\in\mathcal Y$,
\[
    \frac{
        \mathbbm{1}\!\left\{
            y\notin\mathcal C_{\widetilde\gamma}(x)
        \right\}
    }{\widetilde\gamma}
    \geq
    \frac{
        \mathbbm{1}\!\left\{
            y\notin\mathcal C_\gamma(x)
        \right\}
    }{\gamma}.
\]
Thus, moving an atom of the aggregation rule from $\gamma$ to
$\widetilde\gamma$ leaves the corresponding conformal region unchanged
and weakly increases the pointwise incompatibility penalty. Consequently,
the resulting ambiguity set can only shrink, and its robust risk cannot
increase.

When $r=0$, the conformal quantile equals $+\infty$ by convention, so
\[
    \mathcal C_\gamma(x)=\mathcal Y.
\]
Such an atom contributes zero incompatibility. It can therefore be removed
and the remaining weights renormalized without increasing the robust risk.
If it is the only atom, it may be replaced by any candidate level in the
nontrivial rank grid, which is nonempty because $\alpha>1/(m+1)$.

Applying this argument to every atom of any
$\nu\in\mathcal P_K((0,\alpha))$, and merging atoms mapped to the same
location, shows that the infimum in
\eqref{eq:joint-geometry-design} is unchanged when the support of $\nu$
is restricted to
\[
    \Gamma_m
    =
    \left\{
        \frac{r}{m+1}:
        r=1,\ldots,m,\;
        \frac{r}{m+1}<\alpha
    \right\}.
\]
If two candidate levels generate the same conformal region, moving their
combined weight to the smaller level weakly increases the incompatibility
penalty and cannot increase the robust risk. We may therefore remove
duplicates and write
\[
    \Gamma_m
    =
    \{\gamma_1<\cdots<\gamma_R\}.
\]

\paragraph{Step 2: Shell representation on the finite grid.}
Every aggregation rule supported on at most $K$ points of $\Gamma_m$ can
be represented as
\[
    \nu_w
    =
    \sum_{i=1}^R w_i\delta_{\gamma_i},
    \qquad
    w\in\Delta_R,
    \qquad
    \|w\|_0\leq K,
\]
where
\[
    \Delta_R
    =
    \left\{
        w\in\mathbb R_+^R:
        \sum_{i=1}^R w_i=1
    \right\}.
\]

Apply the shell construction from
Section~\ref{sec:tractable-dro} to the full ordered family
\[
    \mathcal C_{\gamma_1}(x)
    \supseteq\cdots\supseteq
    \mathcal C_{\gamma_R}(x).
\]
For every $j\in\mathcal J(x)$ and every
$y\in\mathcal A_j(x)$, the incompatibility penalty is constant and equals
\[
    h_j(w)
    =
    \sum_{i=R-j+1}^{R}\frac{w_i}{\gamma_i},
\]
with the convention $h_0(w)=0$.

Theorem~\ref{thm:shell-reformulation}, applied to this finite shell
partition, therefore gives
\[
    \mathcal R_{\alpha,\nu_w,x}\bigl(c(z,\cdot)\bigr)
    =
    \max_{\{p_j\}_{j\in\mathcal J(x)}}
    \sum_{j\in\mathcal J(x)}p_jL_j(z;x)
\]
subject to
\[
    \sum_{j\in\mathcal J(x)}p_j=1,
    \qquad
    \sum_{j\in\mathcal J(x)}p_jh_j(w)\leq\frac{1}{\alpha},
    \qquad
    p_j\geq0.
\]
Although some coordinates of $w$ may be zero, the full finite-grid
partition remains valid: adjacent shells receiving the same penalty may
be merged without changing the robust value.

Dualizing the incompatibility-budget constraint yields
\[
    \mathcal R_{\alpha,\nu_w,x}\bigl(c(z,\cdot)\bigr)
    =
    \inf_{\lambda\geq0}
    \left\{
        \frac{\lambda}{\alpha}
        +
        \max_{j\in\mathcal J(x)}
        \left[
            L_j(z;x)-\lambda h_j(w)
        \right]
    \right\}.
\]

\paragraph{Step 3: Epigraph reformulation and reparameterization.}
Using the dual form above, problem
\eqref{eq:joint-geometry-design} is equivalent to
\begin{align*}
    \inf_{\substack{
        z\in\mathcal Z(x),\;
        w\in\Delta_R,\\
        \lambda\geq0,\;
        t\in\mathbb R
    }}
    \quad&
    t+\frac{\lambda}{\alpha}
    \\
    \mathrm{s.t.}\quad&
    t+\lambda h_j(w)\geq L_j(z;x),
    \quad j\in\mathcal J(x),
    \\
    &
    \|w\|_0\leq K.
\end{align*}
The only nonlinear terms are the products $\lambda w_i$. Define
\[
    u_i=\lambda w_i,
    \qquad i=1,\ldots,R.
\]
Because $w\in\Delta_R$,
\[
    \sum_{i=1}^R u_i
    =
    \lambda\sum_{i=1}^R w_i
    =
    \lambda.
\]
Moreover,
\[
    \lambda h_j(w)
    =
    \lambda
    \sum_{i=R-j+1}^{R}\frac{w_i}{\gamma_i} \\
    =
    \sum_{i=R-j+1}^{R}\frac{u_i}{\gamma_i},
\]
and
\[
    \|u\|_0\leq\|w\|_0\leq K.
\]
Hence, every feasible solution of the preceding problem produces a
feasible solution of
\[
\begin{aligned}
    \inf_{\substack{
        z\in\mathcal Z(x),\;
        t\in\mathbb R,\;
        u\in\mathbb R_+^R
    }}
    \quad&
    t+\frac{1}{\alpha}\sum_{i=1}^R u_i
    \\
    \mathrm{s.t.}\quad&
    t+
    \sum_{i=R-j+1}^{R}\frac{u_i}{\gamma_i}
    \geq
    L_j(z;x),
    \qquad j\in\mathcal J(x),
    \\
    &
    \|u\|_0\leq K,
\end{aligned}
\]
with the same objective value.

\paragraph{Step 4: Reverse construction and recovery of the weights.}
Conversely, let $(z,t,u)$ be a feasible solution
and define
\[
    \lambda
    =
    \sum_{i=1}^R u_i.
\]
Suppose first that $\lambda>0$ and set
\[
    w_i=\frac{u_i}{\lambda},
    \qquad i=1,\ldots,R.
\]
Then $w\in\Delta_R$ and
\[
    \|w\|_0=\|u\|_0\leq K.
\]
Furthermore,
\[
    \lambda h_j(w)
    =
    \sum_{i=R-j+1}^{R}\frac{u_i}{\gamma_i},
    \qquad j\in\mathcal J(x),
\]
so $(z,w,\lambda,t)$ is feasible for the epigraph formulation and has
the same objective value.

If $\lambda=0$, then $u=0$, and the shell constraints reduce to
\[
    t\geq L_j(z;x),
    \qquad j\in\mathcal J(x).
\]
Choose any $w\in\Delta_R$ satisfying $\|w\|_0\leq K$ and set
$\lambda=0$. The resulting epigraph constraints and objective are
identical. Thus, the two formulations have the same feasible objective
values.

Finally, if $(z^\star,t^\star,u^\star)$ is optimal and
\[
    \lambda^\star
    =
    \sum_{i=1}^R u_i^\star
    >0,
\]
then the inverse transformation gives
\[
    w_i^\star
    =
    \frac{u_i^\star}{\lambda^\star}.
\]
Therefore, an optimal aggregation rule is
\[
    \widehat\nu_{m,K}
    =
    \sum_{\substack{
        i=1,\ldots,R\\
        u_i^\star>0
    }}
    \frac{u_i^\star}{\lambda^\star}
    \delta_{\gamma_i},
\]
which proves the recovery formula.
\end{proof}

\section{Proposed Algorithm}
\label{sec:proposed-algorithm}

Algorithm~\ref{alg:conformal-dro} summarizes an implementation of \texttt{Conformal-DRO} that permits context-dependent aggregation weights. An auxiliary tuning set $\mathcal D_m$ is used to specify a measurable map $x\mapsto\widehat\nu_{m,K}(x)$ under the support budget $K$. The entire map, together with the tuning-sample conformal regions, is then frozen. An independent calibration set $\mathcal D_n$ calibrates the resulting composite outcome score, after which the deployed decision is optimized. This additional calibration is essential: although the map is selected independently of the final calibration sample, its evaluation at $X_{n+1}$ depends on the future context. Therefore, Theorem~\ref{thm:multilevel-validity} does not directly justify using these weights with the original levelwise calibration. The construction below uses split conformal e-prediction to calibrate the incompatibility penalty \citep{vovk2025conformal}; its latent-law guarantee is established below.

We take $\alpha\in(0,1)$, use the candidate level domain $(0,\alpha)$, and assume $\alpha>1/(m+1)$, so that the tuning grid is nonempty. The score $s$ is fixed or fitted using data independent of the tuning sample and the final calibration and future observations. Conditional on the independent training data and $\mathcal D_m$, the final calibration and future observations satisfy the paper's exchangeable latent-pair model, including its conditional outcome sampling assumption. All tuning conventions, including tie-breaking and fallback rules, are fixed before final calibration and yield measurable functions. We assume nonempty decision sets, finite measurable regional losses, and measurable optimizers whenever an attained minimum is used.

\begin{algorithm}[!b]
\caption{\texttt{Conformal-DRO} with context-dependent weights}
\label{alg:conformal-dro}
\small
\begin{algorithmic}[1]
\REQUIRE Independent tuning and calibration samples $\mathcal D_m$ and $\mathcal D_n$; context $x$; nonconformity score $s$; miscoverage level $\alpha>1/(m+1)$; support budget $K\geq1$.
\ENSURE Decision $\widehat z_n(x)$, robust value $\widehat V_n(x)$, and context-dependent aggregation map $\widehat\nu_{m,K}(\cdot)$.

\STATE Using $\mathcal D_m$, construct the tuning conformal regions on the common finite rank grid
$
\Gamma_m
=\{r/(m+1):r=1,\ldots,m,\ r/(m+1)<\alpha\}
=\{\gamma_1<\cdots<\gamma_R\}.
$

\STATE For a generic context $\xi$, construct the finest nonempty tuning shells $\{\mathcal A_j^{\mathrm{tun}}(\xi)\}_{j\in\mathcal J^{\mathrm{tun}}(\xi)}$ and their regional losses $L_j^{\mathrm{tun}}(z;\xi)$ using $\mathcal D_m$.

\STATE Define $(z^\star(\xi),t^\star(\xi),u^\star(\xi))$ by solving
\[
\begin{aligned}
\min_{z\in\mathcal Z(\xi),\,t\in\mathbb R,\,u\in\mathbb R_+^R}
\quad& t+\frac{1}{\alpha}\sum_{i=1}^R u_i\\
\operatorname{s.t.}\quad&
t+\sum_{i=R-j+1}^R\frac{u_i}{\gamma_i}
\geq L_j^{\mathrm{tun}}(z;\xi),
\quad j\in\mathcal J^{\mathrm{tun}}(\xi),\\
&\|u\|_0\leq K.
\end{aligned}
\]

\STATE Set $\lambda^\star(\xi)=\sum_{i=1}^R u_i^\star(\xi)$ and recover
\[
\widehat\nu_{m,K}(\xi)
=\sum_{i:u_i^\star(\xi)>0}
\frac{u_i^\star(\xi)}{\lambda^\star(\xi)}\,\delta_{\gamma_i}
=\sum_{i=1}^R\widehat w_i(\xi)\,\delta_{\gamma_i}.
\]
If $\lambda^\star(\xi)=0$, or the tuning problem has no finite attained optimum, set $\widehat\nu_{m,K}(\xi)=\delta_{\gamma_1}$.

\STATE Freeze the map specified by Steps~2--4 and all tuning regions. Define the incompatibility penalty
\[
h_m(\xi,y)
=\sum_{i=1}^R\frac{\widehat w_i(\xi)}{\gamma_i}
\mathbbm 1\{y\notin\mathcal C_{\gamma_i}(\xi;\mathcal D_m)\}.
\]

\STATE Using only $\mathcal D_n$, evaluate $H_i=h_m(X_i,Y_i)$, $i=1,\ldots,n$, and set $\widetilde{H}_n=\sum_{i=1}^n H_i$. Define $\psi_n(a)=(n+1)a/(\widetilde{H}_n+a)$ for $a\geq0$, with $0/0:=0$.

\STATE At context $x$, form the nonempty shells $\{\mathcal A_j(x)\}_{j\in\mathcal J(x)}$ induced by the selected \emph{tuning} regions. On each shell set $a_j(x)=h_m(x,y)$ for $y\in\mathcal A_j(x)$, $h_j(x)=\psi_n(a_j(x))$, and $L_j(z;x)=\sup_{y\in\mathcal A_j(x)}c(z,y)$. If $\min_{j\in\mathcal J(x)}h_j(x)>1/\alpha$, set every $h_j(x)=0$.

\STATE Obtain $(\widehat z_n(x),\widehat t_n(x),\widehat\lambda_n(x))$ by solving
\[
\begin{aligned}
\min_{z\in\mathcal Z(x),\,t\in\mathbb R,\,\lambda\geq0}
\quad& t+\frac{\lambda}{\alpha}\\
\operatorname{s.t.}\quad&
t+\lambda h_j(x)\geq L_j(z;x),
\qquad j\in\mathcal J(x).
\end{aligned}
\]

\STATE Set $\widehat V_n(x)=\widehat t_n(x)+\widehat\lambda_n(x)/\alpha$.
\end{algorithmic}
\end{algorithm}

The first five steps specify the tuning stage. Step~1 reduces the candidate levels to the finite conformal-rank grid used in Theorem~\ref{thm:cardinality-reformulation}. We retain a common ordered grid across contexts and omit empty shells at each context. Duplicate regions can be retained without affecting the tuning reformulation; if they are merged locally, their combined tuning weight may be assigned to the smallest corresponding level. Steps~2--4 define a tuning routine for a generic input $\xi$, rather than a single aggregation rule for one deployment context. Thus, this same routine can subsequently be evaluated at every calibration context $X_i$ and at the future context $x$. It need not be evaluated at every possible context in advance, but its definition cannot be changed using $\mathcal D_n$ or the particular deployment context.

Step~3 retains the cardinality-constrained tuning formulation in Theorem~\ref{thm:cardinality-reformulation}. At each context,
$$
u_i(\xi)=\lambda(\xi)w_i(\xi),
$$
so $u_i^\star(\xi)>0$ selects level $\gamma_i$, while $\|u^\star(\xi)\|_0\leq K$ limits the number of active levels at that context. When $\lambda^\star(\xi)>0$, Step~4 recovers
$$
\widehat w_i(\xi)
=\frac{u_i^\star(\xi)}{\sum_{\ell=1}^R u_\ell^\star(\xi)},
\qquad
\widehat\nu_{m,K}(\xi)
=\sum_{i=1}^R\widehat w_i(\xi)\delta_{\gamma_i}.
$$
Both the weights and their active support may therefore vary with context. The preliminary decision $z^\star(\xi)$ is used only to select a geometry relevant to the downstream loss. Step~3 is a geometry-selection surrogate: its objective is not the final robust value after composite-score calibration. If the tuning problem is unbounded or fails to attain a finite optimum, the specified one-point fallback still defines an admissible score for calibration.

Step~5 freezes the entire incompatibility penalty $h_m$. It is nonnegative and finite, since
$$
0\leq h_m(\xi,y)
\leq\sum_{i=1}^R\frac{\widehat w_i(\xi)}{\gamma_i}
\leq\frac{1}{\gamma_1}.
$$
No context-conditional moment bound is assumed for this uncalibrated score. Step~6 instead calibrates it as a whole, using each observation's own context in $h_m(X_i,Y_i)$. The resulting outcome penalty before the empty-set fallback is
\begin{equation}
\widetilde h_n(y;x)
=\psi_n\bigl(h_m(x,y)\bigr)
=\frac{(n+1)h_m(x,y)}{\widetilde{H}_n+h_m(x,y)},
\label{eq:contextual-composite-calibration}
\end{equation}
with the same $0/0:=0$ convention. The candidate score in the denominator is part of the calibration rule. The tuning regions and weight map remain fixed; the final sample supplies the scalar $\widetilde{H}_n$, rather than new levelwise conformal regions.

Steps~7--9 perform the final ambiguity-set construction and decision optimization. Because $h_m(x,\cdot)$ is constant on each selected tuning shell, so is $\widetilde h_n(\cdot;x)$. Let $\widehat h_n(y;x)=h_j(x)$ for $y\in\mathcal A_j(x)$, where $h_j(x)$ includes the fallback in Step~7, and define
\begin{equation}
\widehat{\mathcal U}_n^\alpha(x)
=\left\{\mathbb Q\in\mathcal P(\mathcal Y):
\int_{\mathcal Y}\widehat h_n(y;x)\,\mathbb Q(\d y)
\leq\frac{1}{\alpha}\right\}.
\label{eq:contextual-ambiguity-set}
\end{equation}
The preliminary moment set is nonempty exactly when at least one nonempty shell has penalty at most $1/\alpha$. Otherwise Step~7 replaces it by $\mathcal P(\mathcal Y)$ by setting all penalties to zero. This enlargement preserves coverage and provides a conservative fallback without assuming that the innermost conformal region is nonempty.

The shell-mass linear program and its dual continue to apply with these calibrated penalties. Consequently, under the stated finite-loss and attainment conditions, Step~8 yields
$$
\widehat V_n(x)
=\min_{z\in\mathcal Z(x)}
\sup_{\mathbb Q\in\widehat{\mathcal U}_n^\alpha(x)}
\int_{\mathcal Y}c(z,y)\,\mathbb Q(\d y).
$$
Here, $t$ upper-bounds the penalized regional losses and $\lambda$ is the dual variable for the incompatibility budget. Regional suprema must be computed exactly or replaced by certified upper bounds. A measurable feasible solution of the final epigraph problem also provides a valid upper certificate if optimization is approximate, although its objective need not equal the optimal robust value.

To establish finite-sample validity, condition throughout on any independent score-training data and fixed auxiliary randomization. Conditional on $\mathcal D_m$, the scores
$$
H_i=h_m(X_i,Y_i),\qquad i=1,\ldots,n+1,
$$
are exchangeable because the same frozen function is applied to every observation. Write $\widetilde{H}_{n+1}=\sum_{i=1}^{n+1}H_i$. Exchangeability and the $0/0:=0$ convention give
\begin{equation}
\begin{aligned}
\mathbb E\!\left[
\widetilde h_n(Y_{n+1};X_{n+1})\mid\mathcal D_m\right]
&=\mathbb E\!\left[
\frac{(n+1)H_{n+1}}{\widetilde{H}_{n+1}}\mid\mathcal D_m\right]\\
&=\mathbb E\!\left[
\sum_{i=1}^{n+1}\frac{H_i}{\widetilde{H}_{n+1}}\mid\mathcal D_m\right]\\
&=\mathbb P\{\widetilde{H}_{n+1}>0\mid\mathcal D_m\}\leq1.
\end{aligned}
\label{eq:contextual-moment-validity}
\end{equation}
Since the empty-set fallback only decreases the penalty, $0\leq\widehat h_n\leq\widetilde h_n$. The conditional outcome sampling model and the tower property therefore imply
$$
\mathbb E\!\left[
\int_{\mathcal Y}\widehat h_n(y;X_{n+1})\,
\mathbb Q_{n+1}(\d y)\mid\mathcal D_m\right]
=\mathbb E\!\left[
\widehat h_n(Y_{n+1};X_{n+1})\mid\mathcal D_m\right]
\leq1.
$$
Applying Markov's inequality yields
\begin{equation}
\mathbb P\left\{
\mathbb Q_{n+1}\in\widehat{\mathcal U}_n^\alpha(X_{n+1})
\mid\mathcal D_m\right\}\geq1-\alpha.
\label{eq:contextual-law-coverage}
\end{equation}
On this containment event, the robust objective bounds the expected cost for every feasible decision simultaneously. Applying the same containment argument underlying Corollary~\ref{thm:risk-certificate} and then averaging over the tuning and training data gives
\begin{equation}
\mathbb P\left\{
\int_{\mathcal Y}c(\widehat z_n(X_{n+1}),y)\,
\mathbb Q_{n+1}(\d y)
\leq\widehat V_n(X_{n+1})\right\}\geq1-\alpha.
\label{eq:contextual-risk-certificate}
\end{equation}
Thus, $\widehat V_n(X_{n+1})$ provides a finite-sample upper certificate for the expected cost of the selected decision under the future latent law, while allowing context-dependent weights. The guarantee averages over the final calibration sample, future context, and future latent law; it does not assert coverage conditional on each context or on a fixed realized calibration sample.

The support budget $K$ controls the resolution of the tuning geometry at each context. At most $K$ selected levels induce at most $K+1$ nonempty shells, and composite-score calibration adds no new shell boundaries. Larger values of $K$ enlarge the tuning class, but need not monotonically improve the final calibrated robust value or the deployed decision's cost. Since $K$ is an upper bound, fewer than $K$ levels may be used. Finally, $\widehat h_n\leq n+1$, so if $\alpha\leq1/(n+1)$ the calibrated ambiguity set is necessarily $\mathcal P(\mathcal Y)$. Nontrivial exclusions under this normalization therefore require $\alpha>1/(n+1)$.

\section{Reproducibility Details for the Synthetic Experiments in Section~\ref{sec:synthetic-experimental-setup}}
\label{app:synthetic-experimental-details}

This section documents the experimental configuration underlying the synthetic results in Section~\ref{sec:synthetic-experimental-setup}, including the data-generating distributions, data splitting and randomization, conformal calibration, aggregation-rule selection, and finite-sample and population-oracle optimization. It also reports additional shell-allocation results in Figure~\ref{fig:add-pop-syn-linear-shells} and clarifies the scope of the implemented comparisons and evaluation metrics to facilitate accurate interpretation and reproduction.

\paragraph{Data-generating distributions.} Let $Z$ denote an untruncated random variable. The two outcome laws used in the experiment are
\[
Y\sim\mathcal L\!\left(Z\,\middle|\,-5\leq Z\leq5\right),
\qquad
Z\sim\mathcal N(0,1)
\quad\text{or}\quad
Z\sim t_3(\mathrm{loc}=0,\mathrm{scale}=1).
\]
Thus, all densities, distribution functions, quantiles, and simulated observations correspond to the truncated and renormalized law, rather than to an unbounded law whose observations are subsequently clipped. There are no covariates $X$, and hence no conditional data-generating mechanism. The downstream problem uses $c(z,y)=zy$, $\mathcal Y=[-5,5]$, and $\mathcal Z=[0.5,1.5]$. The bounded outcome and decision sets ensure that the linear cost remains bounded.

\paragraph{Data splitting and randomization.} The population-oracle experiments use the known truncated distribution directly and therefore require no simulated observations. Each finite-sample experiment uses one auxiliary tuning sample of size $n_{\mathrm{tune}}=500$ and one independent final calibration sample of size $n_{\mathrm{cal}}=500$. There is no separate training sample because no predictive model is estimated, and there is no test sample. Observations are generated by inverse-transform sampling from the appropriate truncated distribution. For each distribution, the implementation applies \texttt{numpy.random.SeedSequence} to the experiment seed and spawns two random streams, the first for tuning and the second for final calibration. The seed is $12026$ for the Gaussian experiment and $13026$ for the Student-$t_3$ experiment. A single tuning--calibration realization is generated for each distribution and reused for all $K=1,\ldots,100$. The code does not contain a loop over independent Monte Carlo trials.

\paragraph{Prediction rule and conformal calibration.} The predictive center is fixed at zero, and the nonconformity score is $S(y)=|y|$. There is therefore no prediction architecture, loss function, optimizer, regularization parameter, epoch count, or model-training procedure. For a sample of size $n$ and a candidate miscoverage level $\gamma$, the conformal threshold is the order statistic with rank $k_\gamma=\lceil(n+1)(1-\gamma)\rceil$. The resulting outcome region is
\[
\mathcal C_\gamma
=
\left[
\max\{-5,-\widehat q_{1-\gamma}\},
\min\{5,\widehat q_{1-\gamma}\}
\right].
\]
Candidate levels are the conformal rank values $\gamma=r/(n+1)<\alpha$. With $n=500$ and $\alpha=0.2$, this produces the $100$ levels $r=1,\ldots,100$. If tied scores produce identical consecutive regions, the implementation retains the smallest corresponding value of $\gamma$. The saved results contain all $100$ nonduplicate levels for both data-generating distributions.

\paragraph{Selection of the aggregation rule.} For each support budget $K$, a dynamic program selects at most $K$ levels from the tuning-sample conformal path by minimizing the exact cardinality-constrained upper-Riemann objective induced by the positive linear cost. If $\widehat q_k^+$ denotes the upper endpoint associated with the $k$th selected level and $\widehat q_0^+=5$, the unnormalized allocation and its normalized aggregation weights are
\[
u_k
=
\gamma_k\bigl(\widehat q_{k-1}^+-\widehat q_k^+\bigr),
\qquad
w_k
=
\frac{u_k}{\sum_{\ell}u_\ell}.
\]
The selected levels and weights are then held fixed. The final calibration sample is used only to recompute the corresponding conformal regions and their induced shells. This separation ensures that selection of the aggregation rule does not reuse the final calibration observations. Although the optimization allows fewer than $K$ active levels, the saved results select exactly $K$ levels for every $K=1,\ldots,100$ in both distributions.

\paragraph{Finite-sample robust optimization.} For fixed tuned levels and weights, the implementation computes the shellwise suprema of the unit-decision linear loss and solves the resulting shell-mass linear program with the HiGHS backend of \texttt{scipy.optimize.linprog}. It also solves the corresponding epigraph dual independently and records the absolute primal--dual gap. Because the decision interval is strictly positive, the optimized decision is the lower endpoint when the unit robust value is nonnegative and the upper endpoint otherwise. The saved Gaussian and Student-$t_3$ value curves select $z=0.5$ for every $K$. Across the saved finite-sample results, the largest recorded primal--dual gap is approximately $1.07\times10^{-14}$.

\paragraph{Population-oracle optimization.} The population comparator uses the exact equal-tailed regions $\mathcal C_\gamma^\star=[F^{-1}(\gamma/2),F^{-1}(1-\gamma/2)]$ of the truncated distribution. The joint population optimization constrains the levels to $5\times10^{-4}\leq\gamma_1<\cdots<\gamma_K\leq0.1995$, with a minimum adjacent gap of $10^{-5}$. Population allocation results are computed for $K\in\{1,2,4,8,16,32\}$. The nonlinear program is solved by SLSQP with $50$ initializations, a function tolerance of $10^{-10}$, and at most $5000$ iterations per initialization. Four initializations use deterministic linear, quadratic, square-root, and geometric grids of miscoverage levels; the remaining $46$ use randomly generated ordered levels and decisions sampled uniformly from their feasible ranges. The best feasible solution is followed by at most $30$ bounded coordinate-polishing sweeps. A solution is treated as feasible when its maximum constraint violation is at most $2\times10^{-7}$.

For distribution index $d\in\{0,1\}$, with $d=0$ for the Gaussian law and $d=1$ for the Student-$t_3$ law, the population multistart seed is $2026+1000d+K$. Thus, the seeds are $2026+K$ for the Gaussian problems and $3026+K$ for the Student-$t_3$ problems. The population robust-value curve for $K=1,\ldots,100$ is computed separately by the specialized one-dimensional shooting recursion for the positive-linear, equal-tailed design, subject to the same bounds and minimum-gap constraint. This avoids applying the high-dimensional multistart solver separately at every value of $K$.

\paragraph{Comparators and inactive baseline code.} The only comparator executed by the current experiment driver is the population oracle. The module also contains an implementation of a traditional residual-based one-Wasserstein DRO benchmark, but the main experiment does not call it and no Wasserstein baseline result is included in the reported output. If invoked, that implementation would use the calibration observations as the nominal empirical distribution and the tuning observations to select the smallest radius in a user-supplied candidate grid that is at least the empirical one-Wasserstein distance between the two residual samples. The candidate-radius grid is not specified anywhere in the current experiment configuration. Consequently, a Wasserstein comparison cannot be reproduced or reported without an additional radius grid and an explicit call to the baseline routine.

\paragraph{Reported quantities and comparability.} For each $K$, the saved value tables record the population-oracle objective, the finite-sample split-conformal objective and decision, the tuning objective and decision, the selected support size, miscoverage levels, weights, tuning and calibration contour endpoints, sample sizes, random seed, and primal--dual gap. Allocation figures are reported for $K\in\{1,2,4,8\}$, while robust-value curves are reported for every $K=1,\ldots,100$. Within a distribution, all split-conformal values use the same tuning and calibration observations, and all methods use the same outcome law, cost, decision set, value of $\alpha$, and support budget $K$. The population comparator has direct access to the true law and is therefore an oracle reference rather than a data-matched empirical method.

\paragraph{Unimplemented evaluation quantities.} The current code does not define or compute decision regret. In particular, it does not specify whether regret should be measured relative to a population-optimal stochastic decision, a population robust decision with the same $K$, or another oracle, and it does not evaluate decisions on independent test outcomes. The code likewise does not estimate outcome-level coverage, ambiguity-set coverage of a latent law, or coverage of the optimized cost certificate. Since there is no test sample and no repeated-trial loop, neither regret nor coverage is aggregated across trials, and no means, standard errors, confidence intervals, or significance tests are reported. The value $\alpha=0.2$ is the nominal ambiguity-set miscoverage level used by the method; it is not an empirically estimated coverage rate in the present experiment. To add the requested empirical evaluation, one must specify the regret comparator and loss-evaluation rule, the relevant validity event, the test sample size, the number of independent repetitions, and the desired across-trial summaries. Software versions, hardware, and the computational platform are also not recorded by the current implementation.

\paragraph{Additional results}

\begin{figure}[!t]
    \centering
    \begin{subfigure}[t]{\linewidth}
        \centering
        \includegraphics[width=.24\linewidth]{imgs/allocation_linear_truncated_normal_loc0_scale1_ym5_5_K1.pdf}
        \includegraphics[width=.24\linewidth]{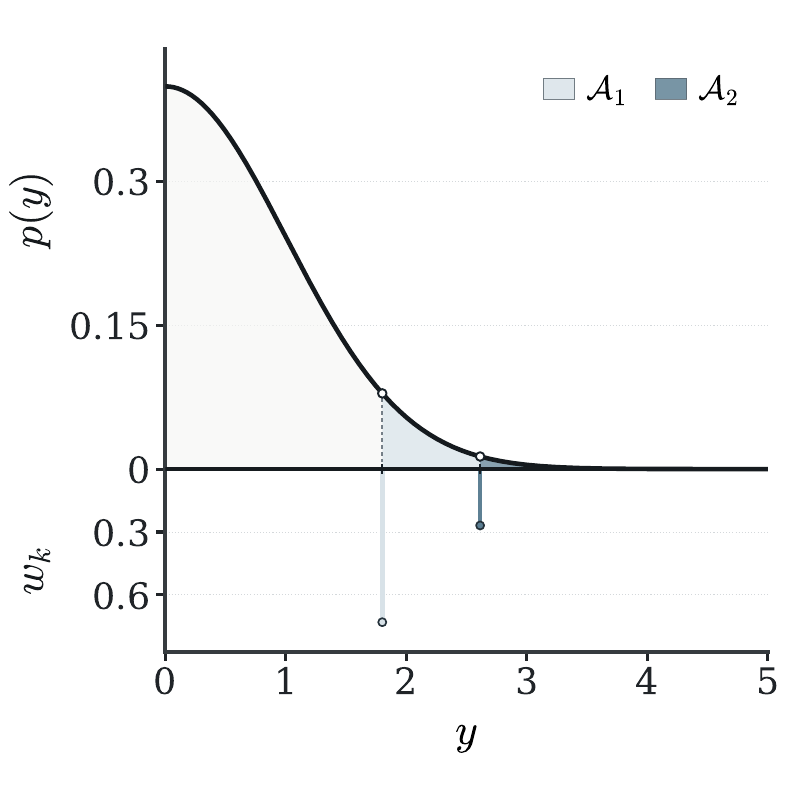}
        \includegraphics[width=.24\linewidth]{imgs/allocation_linear_truncated_normal_loc0_scale1_ym5_5_K4.pdf}
        \includegraphics[width=.24\linewidth]{imgs/allocation_linear_truncated_normal_loc0_scale1_ym5_5_K8.pdf}
        \caption{Population-oracle under Gaussian}
    \end{subfigure}
    \vfill
    \begin{subfigure}[t]{\linewidth}
        \centering
        \includegraphics[width=.24\linewidth]{imgs/allocation_split_conformal_linear_truncated_normal_loc0_scale1_ym5_5_K1.pdf}
        \includegraphics[width=.24\linewidth]{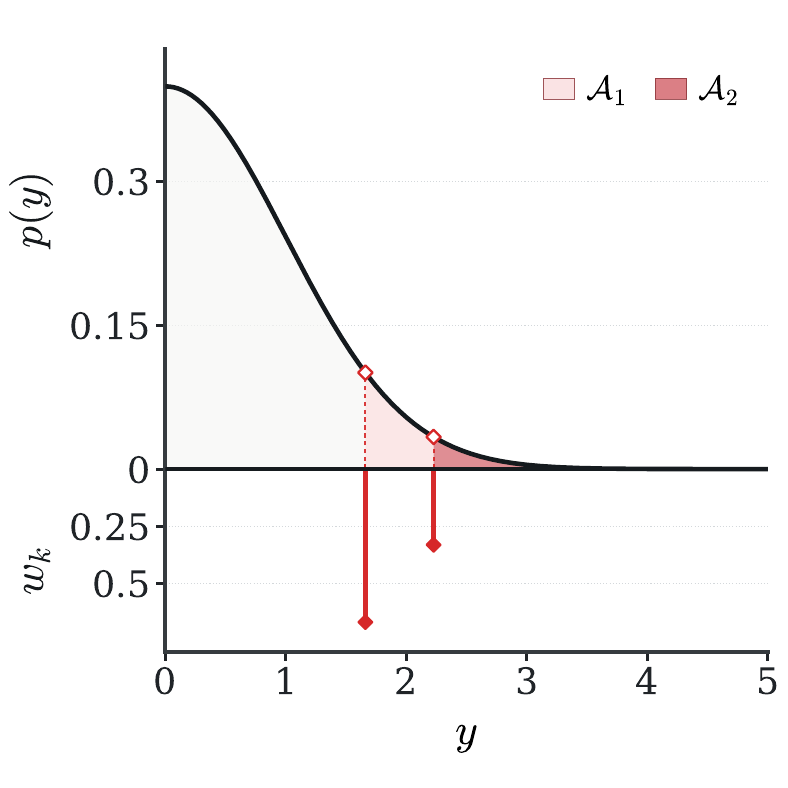}
        \includegraphics[width=.24\linewidth]{imgs/allocation_split_conformal_linear_truncated_normal_loc0_scale1_ym5_5_K4.pdf}
        \includegraphics[width=.24\linewidth]{imgs/allocation_split_conformal_linear_truncated_normal_loc0_scale1_ym5_5_K8.pdf}
        \caption{\texttt{Conformal-DRO} under Gaussian}
    \end{subfigure}
    \vfill
    \begin{subfigure}[t]{\linewidth}
        \centering
        \includegraphics[width=.24\linewidth]{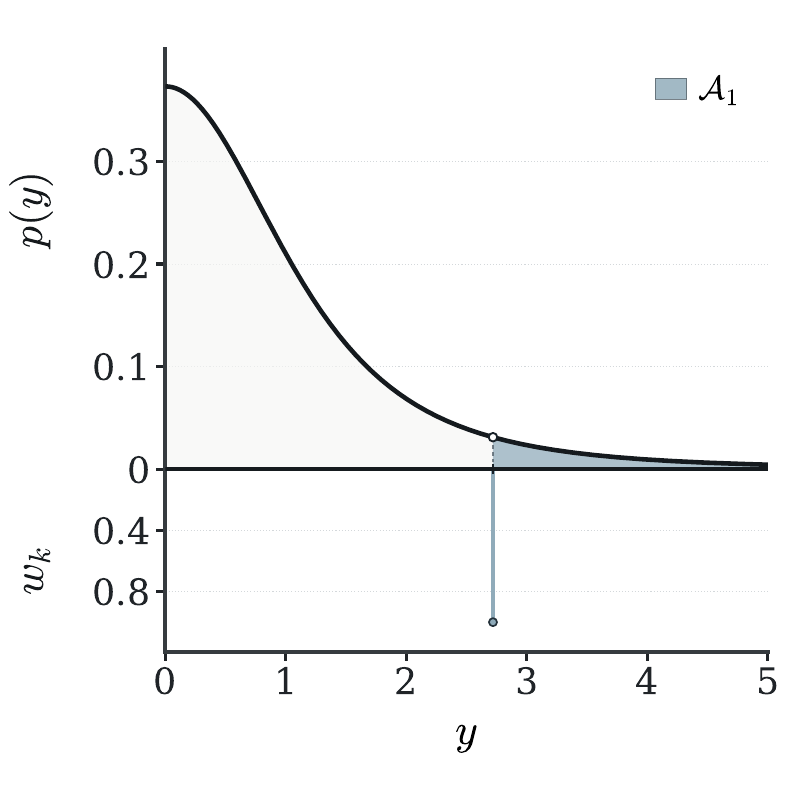}
        \includegraphics[width=.24\linewidth]{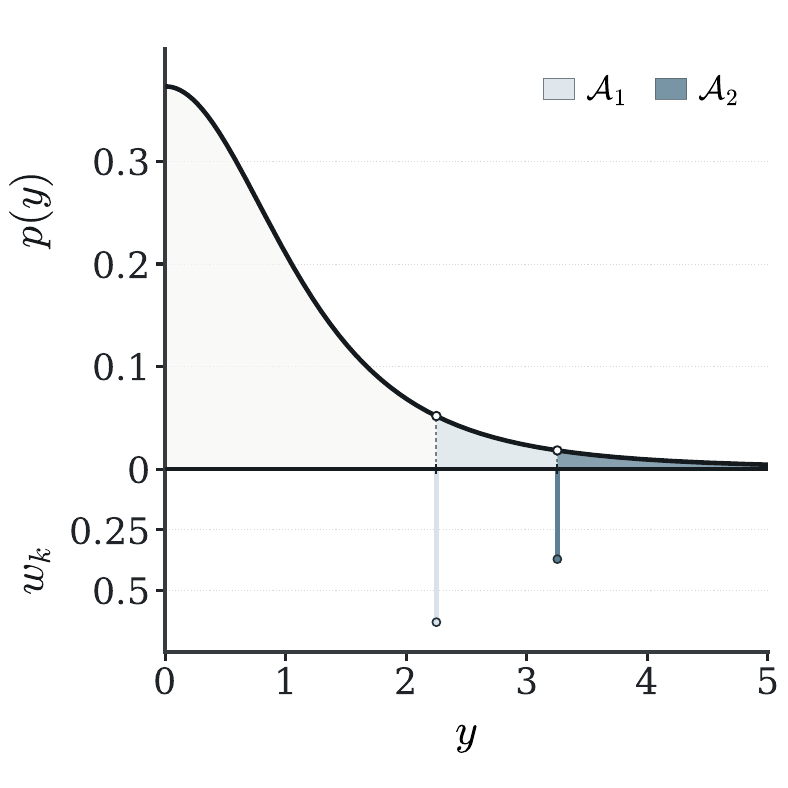}
        \includegraphics[width=.24\linewidth]{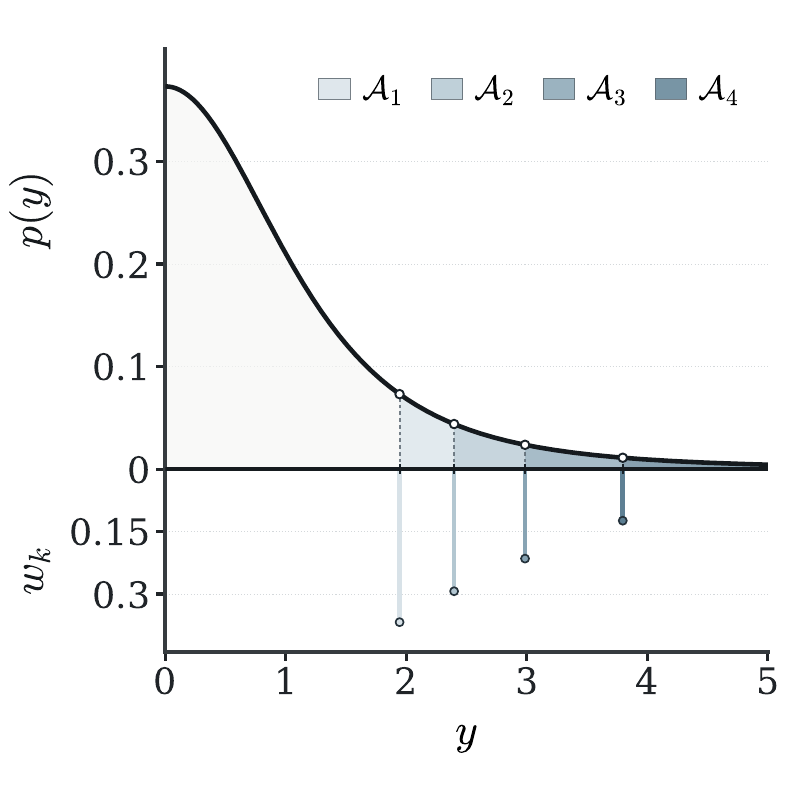}
        \includegraphics[width=.24\linewidth]{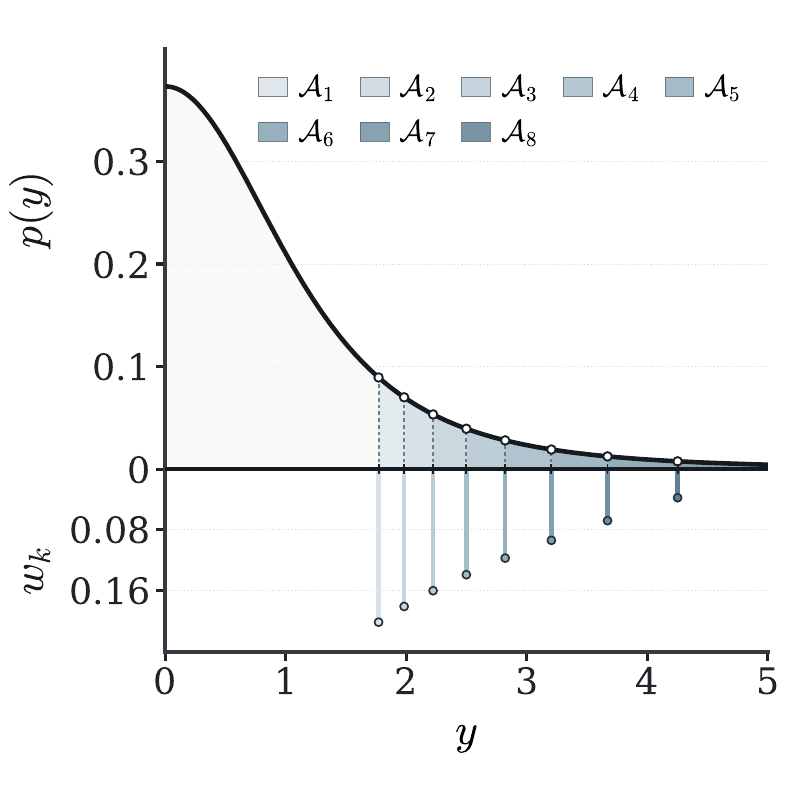}
        \caption{Population-oracle under Student-$t_3$}
    \end{subfigure}
    \vfill
    \begin{subfigure}[t]{\linewidth}
        \centering
        \includegraphics[width=.24\linewidth]{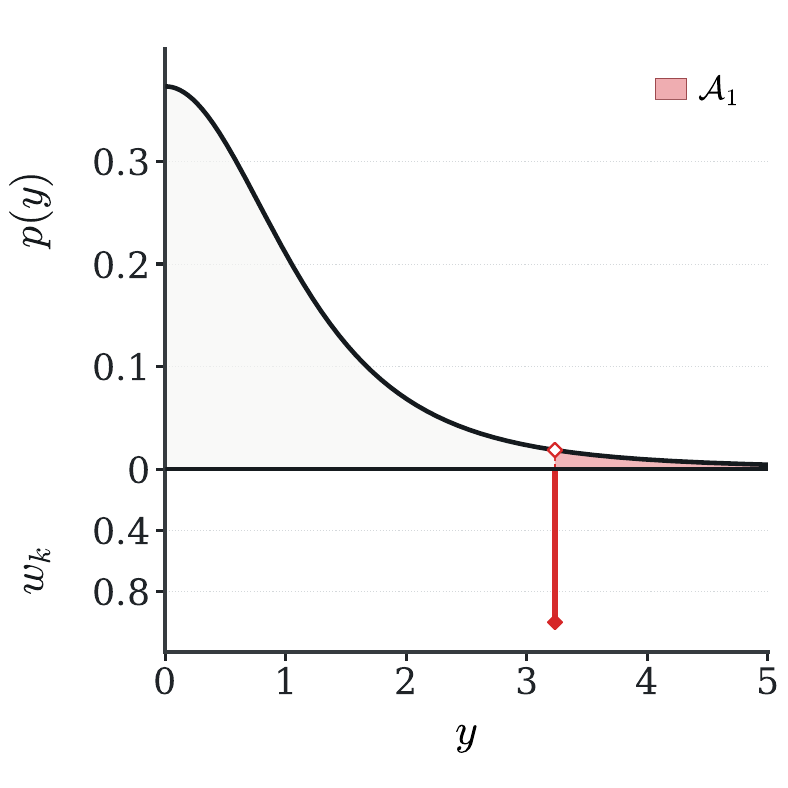}
        \includegraphics[width=.24\linewidth]{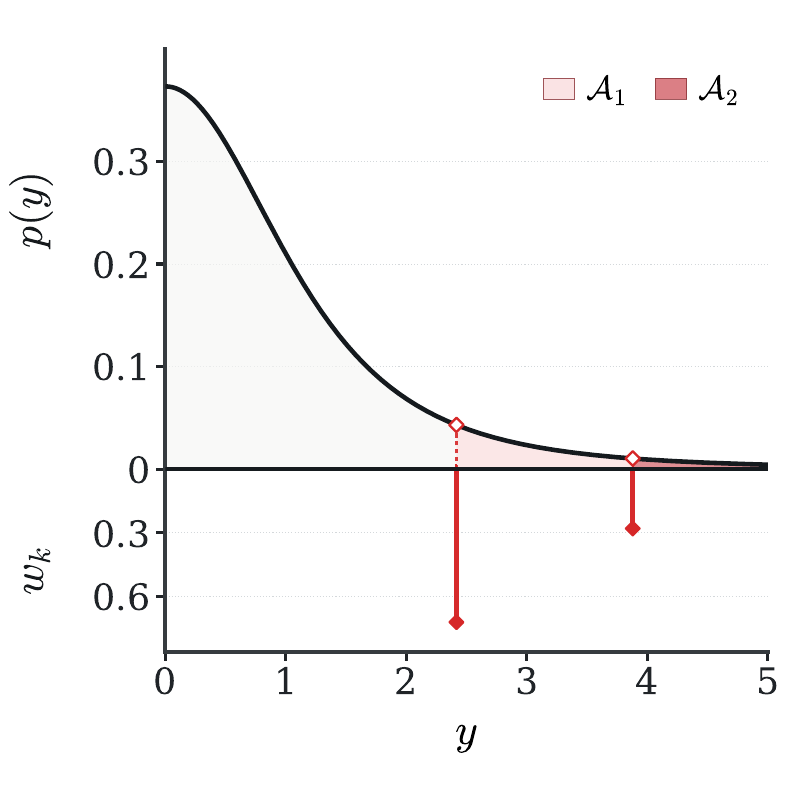}
        \includegraphics[width=.24\linewidth]{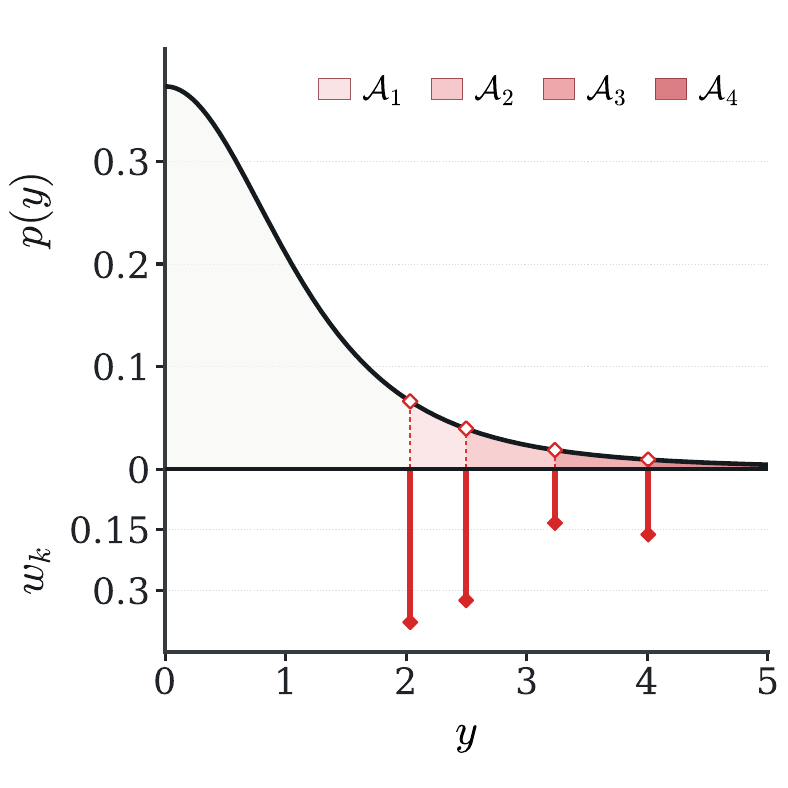}
        \includegraphics[width=.24\linewidth]{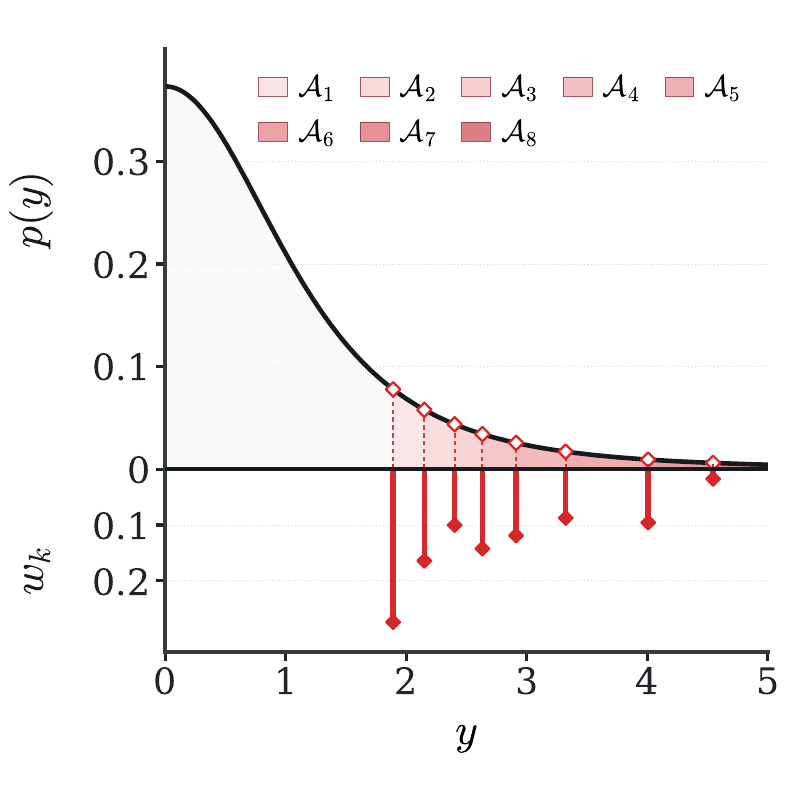}
        \caption{\texttt{Conformal-DRO} under Student-$t_3$}
    \end{subfigure}
    \caption{
    Population-oracle and \texttt{Conformal-DRO} shell allocations under the linear cost $c(z,y)=y^\top z$. 
    Black curves show the population densities, shaded regions the induced shells, dashed lines the optimized upper quantile endpoints, and downward stems the aggregation weights $w_k$. Only $y\geq0$ is shown; $\alpha=0.2$.
    }
    \label{fig:add-pop-syn-linear-shells}
    \vspace{-.1in}
\end{figure}

Figure~\ref{fig:add-pop-syn-linear-shells} compares the population-oracle and finite-sample \texttt{Conformal-DRO} allocations for $K\in\{1,2,4,8\}$. Under both distributions, \texttt{Conformal-DRO} recovers the oracle's qualitative structure: increasing $K$ partitions the upper tail into finer nested shells and distributes weight across more quantile boundaries, with greater emphasis on central and moderately extreme shells. The learned geometry also adapts to tail behavior: the Gaussian boundaries remain relatively concentrated, whereas the Student-$t_3$ boundaries extend farther into the tail because of its slower decay. Finite-sample calibration produces modest shifts in the selected endpoints and weights, particularly for extreme quantiles, but preserves their ordering and overall allocation profile. The main takeaway is that $K$ controls the resolution of the ambiguity geometry, while the outcome distribution determines where the shells are placed and how strongly they are weighted.

\section{Reproducibility Details for the Synthetic Experiments in Section~\ref{subsec:baseline-comparison}}
\label{app:baseline-comparison}

We study a two-product newsvendor problem in which an unobserved covariate changes the demand distribution. The purpose of the experiment is to compare decisions based on the observable mixture distribution with decisions that account for variation in the realized demand law. We consider both discrete and continuous demand models using the same data-generating mechanism and decision problem.

\paragraph{Data generating process.}
Let $X\in[-1,1]^2$ denote the observed context and $W\in\{0,1\}$ an unobserved event indicator:
\begin{equation}
\label{eq:experiment-random-law-model}
X\sim\operatorname{Unif}([-1,1]^2),
\qquad
W\mid X=x\sim\operatorname{Bernoulli}(p(x)),
\end{equation}
where
\[
p(x)
=
\sigma(-1.55+0.9x_1-0.55x_2),
\qquad
\sigma(u)=\frac{1}{1+e^{-u}}.
\]
The marginal event probability is approximately $19\%$. Conditional on $(X_i,W_i)=(x,w)$, the realized demand law is $Y_i\mid(X_i,\mathbb Q_i)\sim \mathbb Q_i=\mathbb Q_{x,w}$. Consequently, conditional on the recorded context alone, $\mathbb Q_i\mid X_i=x\sim\Pi_x\coloneqq [1-\pi(x)]\delta_{\mathbb Q_{x,0}}+\pi(x)\delta_{\mathbb Q_{x,1}}$, whereas the observable conditional response law is $\overline {\mathbb Q}_x=[1-\pi(x)]\mathbb Q_{x,0}+\pi(x)\mathbb Q_{x,1}$. Thus, repeated context-response observations identify the mixture law $\overline {\mathbb Q}_x$, while the law governing the next operating instance remains random because $W_i$ is unobserved.

For both experiments, define
\[
\mu_0(x)
=
\begin{pmatrix}
2.8+0.7x_1+0.3x_2\
3.0-0.4x_1+0.65x_2
\end{pmatrix},
\qquad
a=
\begin{pmatrix}
1.55\
-0.55
\end{pmatrix}.
\]
The vector $\mu_0(x)$ determines the baseline demand level, while $\delta a$ controls how the latent event shifts demand across the two products. We vary $\delta\in\{0,0.5,1,1.5,2\}$ to control latent-law heterogeneity.

In the \emph{discrete experiment}, demand takes values in
$\mathcal Y=\{0,\ldots,6\}^2$.
For a positive-definite matrix A, let
\[
G_{\mu,s}^{A}(y)
=
\frac{
\exp\!\left\{
-\frac{1}{2s^2}(y-\mu)^\top A^{-1}(y-\mu)
\right\}}{
\displaystyle
\sum_{v\in\mathcal Y}
\exp\!\left\{
-\frac{1}{2s^2}(v-\mu)^\top A^{-1}(v-\mu)
\right\}
},
\qquad y\in\mathcal Y,
\]
and set
$\Sigma=
\begin{pmatrix}
0.9^2 & 0.25\\
0.25 & 1
\end{pmatrix}$.
The two latent laws are
\begin{align*}
\mathbb Q_{x,0}
&=
G_{\mu_0(x),s_0(\delta)}^{\Sigma},\quad 
s_0(\delta)=
\max\{0.65,1-0.10\delta\},\\
\mathbb Q_{x,1}
&=
(1-\omega_\delta)
G_{\mu_0(x)+\delta a,s_1(\delta)}^{\Sigma}
+
\omega_\delta
G_{\mu_0(x)+2\delta a+(0.5,0)^\top,s_1(\delta)}^{\Sigma/0.45},
\end{align*}
where
$s_1(\delta)=1+0.55\delta,\ 
\omega_\delta=\min\{0.12\delta,0.28\}$.

In the \emph{continuous experiment}, demand takes values in $[0,6]^2$. Let
$\bar\sigma=(0.72,0.82)^\top,
\ 
s_0(\delta)=\max\{0.65,1-0.10\delta\},
\ 
s_1(\delta)=1+0.42\delta$.
We use truncated Gaussian laws
\begin{align*}
\mathbb Q_{x,0}=
\operatorname{TN}_{[0,6]^2}
\left(
\mu_0(x),
\operatorname{diag}\big((s_0(\delta)\bar\sigma)^2\big)
\right),
\quad
\mathbb Q_{x,1}=
\operatorname{TN}_{[0,6]^2}
\left(
\mu_0(x)+\delta a,
\operatorname{diag}\big((s_1(\delta)\bar\sigma)^2\big)
\right).
\end{align*}

\paragraph{Decision problem and methods.}
The decision is a two-product order quantity
\[
z=(z_1,z_2)\in
\mathcal Z
\coloneqq 
\left\{
z\in\{0,\ldots,6\}^2:
z_1+z_2\leq7
\right\}.
\]
The cost is
\[
c(z,y)
=
\sum_{j=1}^2
\left[
h_j(z_j-y_j)_+
+
b_j(y_j-z_j)_+
\right],
\qquad
h=(1,1.2),
\qquad
b=(6,5.5).
\]
The capacity constraint makes the allocation between the two products depend on the realized demand regime.

We compare the following methods: the observed-W oracle, which minimizes expected cost under the realized law $\mathbb Q_{x,W}$; the mixture stochastic program (Mixture SP), which minimizes expected cost under the exact mixture $\overline {\mathbb Q}_x$; mixture-centered Wasserstein DRO (WDRO); a single-level conformal ambiguity set; and sliced Conformal-DRO with support budgets $K=4$ and $K=8$.

All methods that use the mixture distribution are given the exact $\overline {\mathbb Q}_x$. In particular, the conformal score is
\[
s(x,y)=-\log \bar q_x(y),
\]
where $\bar q_x$ denotes the probability mass function in the discrete experiment and the density in the continuous experiment. This removes conditional-distribution estimation error from the comparison.

For conformal level $\gamma$, let
$\mathcal{C}_\gamma(x)
=
\left\{
y:
s(x,y)\leq\widehat q_{1-\gamma}
\right\}$.
Given an aggregation rule
$\nu=\sum_{k=1}^{K}w_k\delta_{\gamma_k}$,
the conformal ambiguity set is
the corresponding law-level ambiguity set is
\begin{equation}
\label{eq:experiment-conformal-ambiguity}
    \widehat{\mathfrak U}_{\alpha,\nu}(x)
    =
    \left\{
        \mathbb Q\in\mathcal P(\mathcal Y):
        \sum_{k=1}^{K}
        \frac{w_k}{\gamma_k}
        \left(1 - \mathbb Q\!\left(
            \mathcal C_{\gamma_k}(x)
        \right)\right)
        \leq \frac{1}{\alpha}
    \right\}.
\end{equation}
The deployed decision and its robust certificate are
\begin{align*}
    \widehat z_{\alpha,\nu}(x)
    &\in
    \arg\min_{z\in\mathcal Z}
    \sup_{\mathbb Q\in\widehat{\mathfrak U}_{\alpha,\nu}(x)}
    \sum_{y\in\mathcal Y}c(z,y)\mathbb Q(y),
    \\
    \widehat V_{\alpha,\nu}(x)
    &=
    \sup_{\mathbb Q\in\widehat{\mathfrak U}_{\alpha,\nu}(x)}
    \sum_{y\in\mathcal Y}
    c(\widehat z_{\alpha,\nu}(x),y)\mathbb Q(y).
\end{align*}
In the discrete experiment, all quantities are evaluated by finite enumeration. In the continuous experiment, expectations and law probabilities are approximated by deterministic tensor quadrature.

\paragraph{Benchmarks.}
We compare our method with the following benchmarks.

The \emph{observed-$W$ oracle} observes the omitted covariate and minimizes expected cost under $\mathbb Q_{x,W}$. The \emph{mixture stochastic program} knows the exact $\overline {\mathbb Q}_x$ and minimizes $\mathbb E_{\overline {\mathbb Q}_x}[c(z,Y)]$.
The \emph{mixture-centered WDRO} solves
\[
    \min_{z\in\mathcal Z}
    \sup_{\mathbb Q:W_1(\mathbb Q,\overline {\mathbb Q}_x)\leq\varepsilon}
    \mathbb E_{\mathbb Q}[c(z,Y)],
\]
where $W_1$ is induced by the $\ell_1$ ground metric on $\mathcal Y$ and $\varepsilon$ is selected on independent tuning data from $\{0,0.025,0.05,0.1,0.2,0.3,0.5,0.75,1\}$.
The finite-support WDRO problem is solved exactly through its piecewise linear dual \citep{zhang2024optimal}. The \emph{single-level conformal} benchmark restricts~\eqref{eq:experiment-conformal-ambiguity} to $K=1$. The proposed \emph{sliced Conformal-DRO} uses at most four conformal levels ($K=4$) or eight conformal levels ($K=8$).


For each value of $\delta$ and each repetition, we independently generate
tuning sample $\mathcal D_m$ of size $3{,}000$, a final
calibration sample $\mathcal D_n$ of size $1{,}000$, and $2{,}000$ independent
test instances. All samples follow the same
mechanism~\eqref{eq:experiment-random-law-model}.


At $\alpha=0.2$, the candidate conformal levels are
    $\Gamma=\{0.02,0.04,0.06,0.08,0.10,0.12,0.14,0.16,0.18,0.19\}.$ For the single-level method, we consider each $\gamma\in\Gamma$. For sliced Conformal-DRO, we consider combinations of up to K levels. 

To examine the effect of the law-level miscoverage parameter $\alpha$, we fix $\delta=2$ and repeat the experiment. We use
    $\alpha
    \in
    \{0.10,0.15,0.20,0.30,0.40,0.50,0.60\},
    \gamma
    \in
    \alpha\{0.2,0.3,0.4,0.5,0.6,0.8\}.$

\paragraph{Results.}
We evaluate each method by its expected cost under the realized law
\[
    L(X,W)
    \coloneqq 
    \int_{y\in\mathcal Y}
    c(z(X,W),y)\mathbb Q_{X,W}(y),
\]
We report
its mean, $90$th percentile, and $\operatorname{CVaR}_{0.9}$ across future
realized laws, together with mean regret relative to the observed-$W$ oracle
and the mean cost conditional on $W=1$.

For robust methods, we also report the mean robust certificate and certificate coverage:
\[
    \mathbb P\!\left\{
        \mathbb Q_{X,W}
        \in
        \widehat{\mathfrak U}_{\alpha,\nu}(X)
    \right\},
    \mathbb P\!\left\{
        L(X,W)\leq\widehat V(X)
    \right\}.
\]

\begin{table}[t]
\centering
\caption{Discrete-demand results at $\delta=2$ and $\alpha=0.2$. The same ten-level candidate grid is used for the single-level and sliced conformal methods.}
\label{tab:omitted-covariate-discrete}
\scriptsize
\setlength{\tabcolsep}{3.6pt}
\begin{tabular}{lccccc}
\toprule
\multicolumn{6}{c}{\textit{Panel A: Operational performance}}\\
\midrule
Method & Mean cost & P90 cost & $CVaR_{0.9}$ & Mean regret & Event cost\\
\midrule
Observed-$W$ oracle
& $3.408\,(0.034)$ & $6.849\,(0.054)$ & $7.379\,(0.048)$ & $0.000\,(0.000)$ & $6.934\,(0.035)$\\
Mixture SP
& $4.256\,(0.063)$ & $10.072\,(0.051)$ & $12.022\,(0.187)$ & $0.848\,(0.030)$ & $10.569\,(0.104)$\\
Mixture WDRO
& $4.262\,(0.057)$ & $10.020\,(0.056)$ & $11.864\,(0.173)$ & $0.853\,(0.024)$ & $10.409\,(0.144)$\\
\texttt{Conformal-DRO}, $K=1$
& $5.025\,(0.102)$ & $8.071\,(0.189)$ & $10.390\,(0.176)$ & $1.616\,(0.077)$ & $8.465\,(0.114)$\\
\texttt{Conformal-DRO}, $K=4$
& $4.746\,(0.103)$ & $8.106\,(0.344)$ & $10.903\,(0.146)$ & $1.338\,(0.076)$ & $8.790\,(0.096)$\\
\texttt{Conformal-DRO}, $K=8$
& $4.727\,(0.107)$ & $8.040\,(0.474)$ & $10.803\,(0.215)$ & $1.319\,(0.080)$ & $8.741\,(0.171)$\\
\bottomrule
\end{tabular}

\vspace{0.65em}

\begin{tabular}{lccc}
\toprule
\multicolumn{4}{c}{\textit{Panel B: Reliability and reported values}}\\
\midrule
Method & Mean certificate & Law coverage & Certificate coverage\\
\midrule
Mixture SP
& $4.205\,(0.009)$ & $0.000\,(0.000)^{\dagger}$ & $0.804\,(0.008)$\\
Mixture WDRO
& $7.206\,(1.442)$ & $0.345\,(0.232)$ & $0.849\,(0.044)$\\
\texttt{Conformal-DRO}, $K=1$
& $25.844\,(0.889)$ & $0.927\,(0.018)$ & $1.000\,(0.000)$\\
\texttt{Conformal-DRO}, $K=4$
& $21.162\,(0.954)$ & $0.946\,(0.004)$ & $1.000\,(0.000)$\\
\texttt{Conformal-DRO}, $K=8$
& $19.665\,(0.696)$ & $0.945\,(0.004)$ & $1.000\,(0.000)$\\
\bottomrule
\end{tabular}

\vspace{0.25em}
\begin{minipage}{0.95\linewidth}
\footnotesize
${}^{\dagger}$For Mixture SP, law coverage is a diagnostic quantity obtained by treating the exact mixture law as the singleton nominal set $\{\overline Q_x\}$. Mixture SP itself does not claim law-level coverage. Its ``certificate'' is the exact mixture objective, whereas the Mixture WDRO certificate is the Wasserstein-robust value.
\end{minipage}
\end{table}

\begin{table}[t]
\centering
\caption{Continuous truncated-Gaussian-mixture results at $\delta=2$ and $\alpha=0.2$. Expectations and law probabilities are evaluated by deterministic tensor quadrature.}
\label{tab:omitted-covariate-continuous}
\scriptsize
\setlength{\tabcolsep}{3.6pt}
\begin{tabular}{lccccc}
\toprule
\multicolumn{6}{c}{\textit{Panel A: Operational performance}}\\
\midrule
Method & Mean cost & P90 cost & $CVaR_{0.9}$ & Mean regret & Event cost\\
\midrule
Observed-$W$ oracle
& $3.060\,(0.019)$ & $5.477\,(0.038)$ & $6.318\,(0.052)$ & $0.000\,(0.000)$ & $5.665\,(0.032)$\\
Mixture SP
& $3.933\,(0.037)$ & $8.888\,(0.054)$ & $10.996\,(0.123)$ & $0.872\,(0.021)$ & $9.251\,(0.071)$\\
Mixture WDRO
& $3.935\,(0.036)$ & $8.831\,(0.078)$ & $10.802\,(0.217)$ & $0.874\,(0.020)$ & $9.112\,(0.132)$\\
\texttt{Conformal-DRO}, $K=1$
& $4.778\,(0.057)$ & $7.820\,(0.128)$ & $10.525\,(0.348)$ & $1.718\,(0.056)$ & $7.494\,(0.214)$\\
\texttt{Conformal-DRO}, $K=4$
& $4.519\,(0.071)$ & $7.533\,(0.079)$ & $10.021\,(0.314)$ & $1.459\,(0.068)$ & $7.474\,(0.158)$\\
\texttt{Conformal-DRO}, $K=8$
& $4.495\,(0.065)$ & $7.484\,(0.127)$ & $9.975\,(0.302)$ & $1.434\,(0.058)$ & $7.470\,(0.153)$\\
\bottomrule
\end{tabular}

\vspace{0.65em}

\begin{tabular}{lccc}
\toprule
\multicolumn{4}{c}{\textit{Panel B: Reliability and reported values}}\\
\midrule
Method & Mean certificate & Law coverage & Certificate coverage\\
\midrule
Mixture SP
& $3.893\,(0.008)$ & $0.000\,(0.000)^{\dagger}$ & $0.802\,(0.006)$\\
Mixture WDRO
& $6.962\,(1.864)$ & $0.406\,(0.274)$ & $0.879\,(0.056)$\\
\texttt{Conformal-DRO}, $K=1$
& $23.000\,(1.256)$ & $0.937\,(0.006)$ & $1.000\,(0.000)$\\
\texttt{Conformal-DRO}, $K=4$
& $21.428\,(0.686)$ & $0.944\,(0.004)$ & $1.000\,(0.000)$\\
\texttt{Conformal-DRO}, $K=8$
& $20.901\,(0.775)$ & $0.942\,(0.005)$ & $1.000\,(0.000)$\\
\bottomrule
\end{tabular}

\vspace{0.25em}
\begin{minipage}{0.95\linewidth}
\footnotesize
${}^{\dagger}$The Mixture SP coverage entry again treats $\{\overline {\mathbb Q}_x\}$ as a diagnostic singleton nominal set. For Mixture SP and Mixture WDRO, certificate coverage is diagnostic; the formal realized-law guarantee applies only to the conformal ambiguity sets.
\end{minipage}
\end{table}

To examine the protection--efficiency role of $\alpha$, we fix the discrete design at $\delta=2$, use $K=4$, and repeat the full tuning--calibration--testing procedure at $\alpha=0.10,0.15,0.20,0.30,0.40,0.50$, and $0.60$. At each value of $\alpha$, the candidate levels equal $\alpha$ times the ten fractions $0.10,0.20,\ldots,0.90$, and $0.95$, so the relative contour resolution remains fixed. In particular, the grid at $\alpha=0.2$ is the one used in Table~\ref{tab:omitted-covariate-discrete}. Parentheses in all tables contain Student-$t$ $95\%$ confidence-interval half-widths across the ten repetitions.

\begin{table}[t]
\centering
\caption{Protection--efficiency for sliced Conformal-DRO with $K=4$ in the discrete-demand experiment at $\delta=2$. Mixture SP is included as the oracle benchmark for mean performance.}
\label{tab:omitted-covariate-alpha-frontier}
\scriptsize
\setlength{\tabcolsep}{3.8pt}
\begin{tabular}{lccccc}
\toprule
Method / $\alpha$ & Mean cost & P90 cost & $CVaR_{0.9}$ & Law coverage & Mean certificate\\
\midrule
Sliced / $0.10$ & $5.292\,(0.099)$ & $7.740\,(0.166)$ & $9.788\,(0.142)$ & $0.979\,(0.003)$ & $22.544\,(0.657)$\\
Sliced / $0.15$ & $4.987\,(0.061)$ & $7.506\,(0.115)$ & $10.196\,(0.289)$ & $0.968\,(0.008)$ & $21.593\,(0.835)$\\
Sliced / $0.20$ & $4.746\,(0.103)$ & $8.106\,(0.344)$ & $10.903\,(0.146)$ & $0.946\,(0.004)$ & $21.162\,(0.954)$\\
Sliced / $0.30$ & $4.354\,(0.072)$ & $9.798\,(0.122)$ & $11.514\,(0.249)$ & $0.927\,(0.010)$ & $20.484\,(0.854)$\\
Sliced / $0.40$ & $4.325\,(0.067)$ & $9.954\,(0.058)$ & $11.746\,(0.297)$ & $0.845\,(0.020)$ & $20.080\,(2.320)$\\
Sliced / $0.50$ & $4.339\,(0.080)$ & $9.938\,(0.108)$ & $12.021\,(0.166)$ & $0.824\,(0.011)$ & $16.717\,(1.241)$\\
Sliced / $0.60$ & $4.329\,(0.080)$ & $9.888\,(0.215)$ & $12.142\,(0.368)$ & $0.815\,(0.011)$ & $15.196\,(1.366)$\\
\midrule
Mixture SP & $4.256\,(0.063)$ & $10.072\,(0.051)$ & $12.022\,(0.187)$ & -- & $4.205\,(0.009)$\\
\bottomrule
\end{tabular}
\end{table}

\paragraph{Observations.}
The negative controls first verify that the comparison is driven by latent-law heterogeneity rather than by an intrinsic advantage of the conformal formulation. When $\delta=0$, $\mathbb Q_{x,0}=\mathbb Q_{x,1}=\overline {\mathbb Q}_x$, observing $W$ has no value, and the random-law hierarchy collapses. In the discrete experiment, the observed-$W$ oracle and Mixture SP both have mean cost $3.367\,(0.006)$; the corresponding values are $3.373\,(0.009)$ for Mixture WDRO, $3.416\,(0.010)$ for the single-level method, $3.399\,(0.012)$ for $K=4$, and $3.394\,(0.011)$ for $K=8$. In the continuous experiment, the oracle and Mixture SP both have mean cost $2.995\,(0.006)$, compared with $3.000\,(0.009)$ for Mixture WDRO, $3.012\,(0.006)$ for the single-level method, and $3.016\,(0.009)$ for either sliced specification. The performance distinctions therefore emerge when the omitted covariate changes the response law and largely disappear when it does not.

Under strong heterogeneity, the two outcome models deliver the same central conclusion. Mixture SP is mean-optimal among policies measurable with respect to the recorded context, but this average objective masks substantial variation across future realized laws. At $\delta=2$, sliced Conformal-DRO with $K=4$ incurs an $11.5\%$ mean-cost premium over Mixture SP in the discrete experiment and a $14.9\%$ premium in the continuous experiment. In return, it lowers P90 realized-law cost by $19.5\%$ and $15.2\%$, $CVaR_{0.9}$ by $9.3\%$ and $8.9\%$, and event-regime cost by $16.8\%$ and $19.2\%$, respectively. Since Mixture SP and Mixture WDRO receive the exact $\overline {\mathbb Q}_x$, these gaps reflect the distinction between the mixture target and the future realized-law target rather than conditional-distribution estimation error.

The law-coverage diagnostics reinforce this interpretation. At $\delta=2$, the singleton mixture nominal set contains none of the future realized laws in either experiment. The mean-cost-tuned Wasserstein balls contain $34.5\%$ of realized laws in the discrete design and $40.6\%$ in the continuous design, with substantial variation across repetitions. By contrast, the conformal methods attain approximately $94\%$--$96\%$ law coverage and $100\%$ certificate coverage. These empirical coverages exceed the nominal level $1-\alpha=80\%$, revealing conservatism in the response-to-law conversion. Slicing should therefore be interpreted as reducing the operational price of law-level protection, rather than as creating additional coverage for its own sake.

Relative to a single conformal level, slicing consistently reduces average cost and certificate size, although its tail effect differs across the two outcome models. With $K=4$, mean cost falls by $5.5\%$ in the discrete experiment and $5.4\%$ in the continuous experiment; the corresponding mean certificates fall by $18.1\%$ and $6.8\%$. In the continuous experiment, slicing also lowers P90 and $CVaR_{0.9}$ while leaving event cost essentially unchanged. In the discrete experiment, P90 is essentially unchanged, whereas $CVaR_{0.9}$ and event cost increase moderately as the sliced rule shifts capacity back toward the more common quiet regime; law coverage decreases from $0.960$ to $0.946$ but remains well above the nominal target. Thus, the evidence supports a reduction in the average and certificate cost of law-level protection, not uniform dominance in every tail statistic.

Increasing the support budget from $K=4$ to $K=8$ produces only small additional changes in realized operational performance. The paired $K=8$ minus $K=4$ differences in mean cost are $-0.019\,(0.047)$ and $-0.024\,(0.033)$ in the discrete and continuous experiments, respectively; the corresponding P90 differences are $-0.066\,(0.536)$ and $-0.049\,(0.096)$. Law coverage is virtually unchanged. The clearer remaining benefit is certificate refinement: the mean certificate decreases by $7.1\%$ in the discrete experiment and $2.5\%$ in the continuous experiment. Four levels therefore capture most of the operational value of the conformal path in these designs, while additional levels primarily refine the ambiguity geometry.

Finally, Table~\ref{tab:omitted-covariate-alpha-frontier} shows that the conservatism is tunable through $\alpha$. Increasing $\alpha$ moves the decision toward mixture-optimal mean performance while gradually giving up tail protection and law coverage. At $\alpha=0.30$, for example, sliced Conformal-DRO has mean cost $4.354$, only $2.3\%$ above Mixture SP, while its P90 and $CVaR_{0.9}$ remain $2.7\%$ and $4.2\%$ lower; empirical law coverage is $0.927$. The frontier is not exactly monotone because the action set is discrete and the aggregation rule is re-tuned at each value of $\alpha$, but it clearly exhibits the intended protection--efficiency tradeoff.

\end{document}